\documentclass[11pt]{article}

\usepackage[utf8]{inputenc}
\usepackage{scalerel}
\usepackage{enumerate}
\usepackage{stmaryrd}
\usepackage{authblk}
\usepackage{wrapfig}
\usepackage{tikz,tikz-cd}
\usepackage{amsmath,amscd,amssymb}
\usepackage{amsthm}
\usepackage{amsfonts}
\usepackage{mathtools}
\usepackage{color}
\usepackage[mathscr]{eucal}
\usepackage{amssymb}
\usepackage[margin=1in]{geometry}
\usepackage[]{lineno}
\usepackage{subcaption}
\usepackage{overpic}
\usepackage{mathtools}
\usepackage{url}
 
 \newtheorem{theorem}{Theorem}[section]
 \newtheorem{proposition}[theorem]{Proposition}
 \newtheorem{corollary}[theorem]{Corollary}
 \newtheorem{lemma}[theorem]{Lemma}

 \theoremstyle{definition}
 \newtheorem{definition}[theorem]{Definition}
 \newtheorem{example}[theorem]{Example}
  \newtheorem{remark}[theorem]{Remark}

\newcommand{\real}{\ensuremath{\mathbb{R}}}
\newcommand{\borel}[1]{\ensuremath{\mathcal{B}(#1)}}
\newcommand{\pdist}[1]{\ensuremath{d_{p,\mathcal{#1}}}}
\newcommand{\gpdist}{\ensuremath{d_p}}
\newcommand{\mdist}[1]{\ensuremath{m_{p,\mathcal{#1}}}}
\newcommand{\bardist}[1]{\ensuremath{\hat{d}_{p,\mathcal{#1}}}}
\newcommand{\mm}[1]{\ensuremath{\mathcal{#1}}}
\newcommand{\ct}{\ensuremath{\sigma}}
\newcommand{\qct}{\ensuremath{\kappa}}
\newcommand{\kct}{\ensuremath{\sigma_{k,p}}}
\newcommand{\mpt}{\ensuremath{\Lambda}}
\newcommand{\lset}[1]{\ensuremath{\downarrow^{#1}}}

\newcommand{\dist}[1]{\ensuremath{d_{p,\mathcal{#1}}}}
\newcommand{\tree}[1]{\ensuremath{T_p(\mathcal{#1})}}

\newcommand{\ftree}[1]{\ensuremath{\mathcal{F}_p(\mathcal{#1})}}
\newcommand{\kftree}[1]{\ensuremath{\mathcal{F}_{k,p}(\mathcal{#1})}}

\newcommand{\qnorm}[1]{\ensuremath{\|{#1}\|_{\nu,q}}}
\newcommand{\ddist}{\ensuremath{\theta_{k,q}}}
\newcommand{\fdist}{\ensuremath{\theta}}
\newcommand{\dgw}{\ensuremath{d_{GW,p}}}
\newcommand{\dks}{\ensuremath{d_{KS,p}}}
\newcommand{\dkm}{\ensuremath{d_{GW,p}}}

\newcommand{\fks}{\ensuremath{D_{KS,p}}}
\newcommand{\var}{\ensuremath{V_p}}
\newcommand{\qV}{\ensuremath{W_p}}
\newcommand{\cdist}[1]{\ensuremath{d_{p,\mathcal{#1}_n}}}
\newcommand{\cct}[1]{\ensuremath{\sigma_{p,\mathcal{#1}_n}}}
\newcommand{\cftree}[1]{\ensuremath{\mathcal{F}_p({\mathcal{#1}}_n)}}
\newcommand{\ctree}[1]{\ensuremath{T_p({\mathcal{#1}}_n)}}
\newcommand{\cqmap}[1]{\ensuremath{\alpha_{p,\mathcal{#1}_n}}}
\newcommand{\cqct}[1]{\ensuremath{\kappa_{p,\mathcal{#1}_n}}}
\newcommand{\bcct}{\ensuremath{\tau_\varepsilon}}

\newcommand{\bcftree}[1]{\ensuremath{\mathcal{F}_{p,\varepsilon}({\mathcal{#1}}_n)}}

\begin{document}


\title{Robust Representation and Estimation of Barycenters and Modes of Probability Measures on Metric Spaces}
\author{Washington Mio and Tom Needham}
\affil{Department of Mathematics, Florida State University, Tallahassee, FL, USA}
\date{ }
\maketitle

\begin{abstract}
This paper is concerned with the problem of defining and estimating statistics for distributions on spaces such as Riemannian manifolds and more general metric spaces. The challenge comes, in part, from the fact that statistics such as means and modes may be unstable: for example, a small perturbation to a distribution can lead to a large change in Fr\'echet means on spaces as simple as a circle. We address this issue by introducing a new merge tree representation of barycenters called the barycentric merge tree (BMT), which takes the form of a measured metric graph and summarizes features of the distribution in a multiscale manner. Modes are treated as special cases of barycenters through diffusion distances. In contrast to the properties of classical means and modes, we prove that BMTs are stable---this is quantified as a Lipschitz estimate involving optimal transport metrics. This stability allows us to derive a consistency result for approximating BMTs from empirical measures, with explicit convergence rates. We also give a provably accurate method for discretely approximating the BMT construction and use this to provide numerical examples for distributions on spheres and shape spaces.
\end{abstract}

\medskip
{\em Keywords:} Fr\'{e}chet mean, barycenter, median, modes of probability distributions, merge trees.

\medskip
{\em 2020 Mathematics Subject Classification:} 62R20, 62R30, 62R40, 55N31

\section{Introduction}

The core theme of this paper is the development of stable and robust representations for barycenters and modes of probability distributions on metric spaces. The primary goal is to construct simple and provably stable summaries of barycenters and modes for datasets formed of objects that can be of varied nature, so long as they are representable as points in a metric space. This is achieved through a summary representation termed {\em barycentric merge tree} (BMT).

\subsection{Barycenters}

The mean of a probability measure $\mu$ on Euclidean space $\real^d$, as the expected value of a random variable $x \in \real^d$ with law $\mu$, has long been used as a simple statistical summary that can be robustly estimated from sufficiently many independent draws from $\mu$. The reinterpretation of the mean as the most central value of $x$, as measured by variance, dates back at least to \'{E}.\ Cartan in studies of more general forms of centers of mass of distributions on Riemannian manifolds (cf.\,\cite[p.\,235]{berger2003view}). If $\mu$ has finite second moment, the mean is the unique minimizer of the Fr\'{e}chet variance function $V_2 \colon \real^d \to \real$ given by
\begin{equation}
V_2(x) \coloneqq \int_{\real^d} \|y-x\|^2 d\mu (y),
\end{equation}      
the expected value of the squared distance to $x$ \cite{frechet1948}. This viewpoint shed light on how to define barycenters for probability distributions on general metric spaces. More formally, a {\em metric measure space} ($mm$-space) is a triple $\mm{X} = (X, d_X, \mu)$, where $(X,d_X)$ is a Polish (complete and separable) metric space and $\mu$ is a Borel probability measure on $(X,d_X)$. The measure $\mu$ \emph{has finite $p$-moments}, $p \geq 1$, if $\int_X d^p (x_0,y) \, d\mu (y) < \infty$ for some (and therefore all) $x_0 \in X$. The collection of all Borel probability measures on $(X,d_X)$ with finite $p$-moments is denoted $\borel{X,d_X;p}$. For $\mu \in \borel{X,d_X;p}$, the {\em Fr\'{e}chet variance function} of order $p$ of $\mm{X}$, denoted $\var \colon X \to \real$, is defined by
\begin{equation} \label{E:mdistance}
\var(x) \coloneqq \mathbb{E}_\mu [d_X^p(x,\cdot)] = \int_X d_X^p (x,y) \,d\mu (y) \,.
\end{equation}
The set of global minima of $V_p$ (the global minimum may not be unique---see Example \ref{E:circle}) is known as the {\em $p$-barycenter} of $\mm{X}$, also as the {\em Fr\'{e}chet mean} for $p=2$ and the {\em Fr\'{e}chet median} for $p=1$. 
As in \cite{hang2019ffunction}, we replace $V_p$ with its $p$th root $\ct_p \colon X \to \real$, $\ct_p (x) \coloneqq (V_p (x))^{1/p}$, and call it the {\em $p$-deviation function} of $\mm{X}$. Clearly, the minimum sets of $V_p$ and $\ct_p$ are the same, so we view $p$-barycenters as minimizers of $\ct_p$. This is done because $\ct_p$ has better analytical properties than $V_p$. As local minimum sets are also important in our formulation, henceforth we relax the terminology to also include the local minima of $V_p$ in the $p$-barycenter of $\mm{X}$.

Barycenters have received significant attention in the literature. The existence and uniqueness of global minima of $V_2$ on Hadamard manifolds (that is, complete, non-positively curved, simply-connected Riemmanian manifolds) was already known to Cartan, whereas local minima of $V_2$ were first investigated in the Riemmanian setting by Grove and Karcher \cite{grka1973}, with subsequent developments in \cite{grkaruh1974,grkaruh1974a,karcher1977}. Barycenters for probability distributions on Wasserstein space have been investigated in \cite{agueh2011,kim2017}. Other studies of barycenters include \cite{batpat2003,batpat2005,afsari2011,le2001,arnaudon2014,huckemann2024}, and \cite{hang2019ffunction} provides a persistent homology take on $p$-deviation functions. For additional historical remarks on developments related to barycenters, including applications, the reader may consult \cite{afsari2011,huckemann2024}. 

\begin{wrapfigure}[14]{r}{0.5\textwidth} 
\vspace{-30pt}
  \begin{center}
    \includegraphics[width=0.48\textwidth]{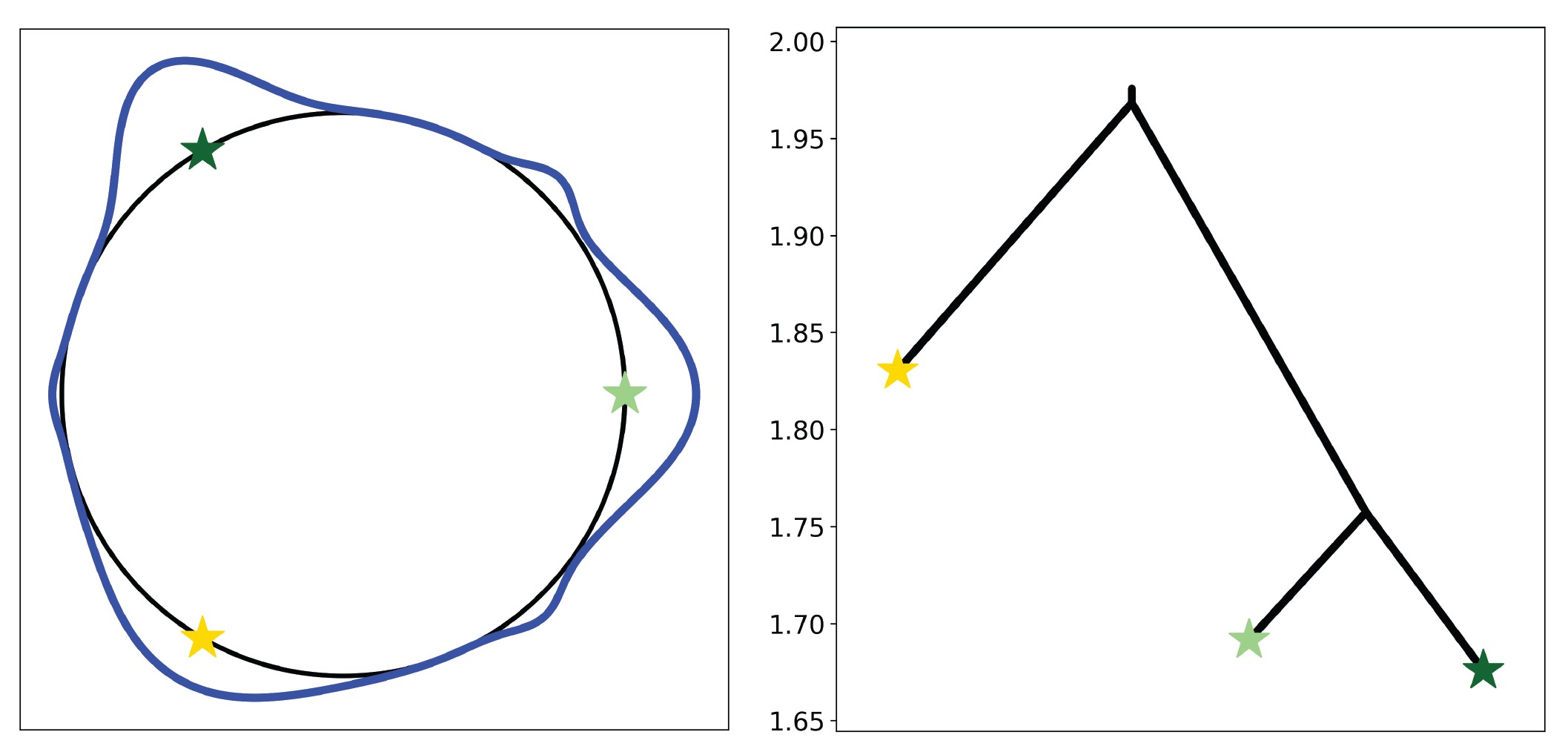}
    \caption{\small Barycentric merge tree example. {\bf Left:} A probability density on the circle. The points marked by $\star$ are local minima of $\sigma_2$, with the darkest indicating the (global) Fr\'{e}chet mean. {\bf Right:} The barycentric merge tree summarizes the local minima, together with the shape of deviation function.}
    \label{fig:BMT_example}
  \end{center}
  \vspace{-17pt}
\end{wrapfigure} 

In spite of the above, as is illustrated in Examples \ref{E:circle} and \ref{E:circle_median}, the global minima of $\ct_p$ (or equivalently, $V_p$) can exhibit unstable behavior that makes them unsuited for data representation and analysis. Although not as pronounced, this instability persists even if we view the barycenter as comprising all local minima of $\ct_p$. For this reason, this paper seeks to construct a representation of barycenters that is as simple as possible and yet robust, with guarantees of stability and consistency. This leads us to refine barycenters and define richer {\em barycentric merge tree (BMT)} representations of $\mm{X}$, show that they are stable, prove a consistency result for BMTs, and map a pathway to discrete models and computation. Roughly, the structure of the BMT representation of $\mathcal{X}$ is as follows: the BMT is a rooted tree whose leaves represent the connected components of the local minimum sets of $\ct_p$ (i.e., of the $p$-barycenter), such that the tree structure captures interrelationships and connectivity properties of the sets of global and local minima. A BMT thus provides a more complete statistical summary of the $mm$-space $(X,d_X,\mu)$ than just the barycenter---see Figure~\ref{fig:BMT_example} for an example on the circle.

\begin{figure}
    \centering
    \begin{overpic}[abs,unit=1mm,scale=.305]{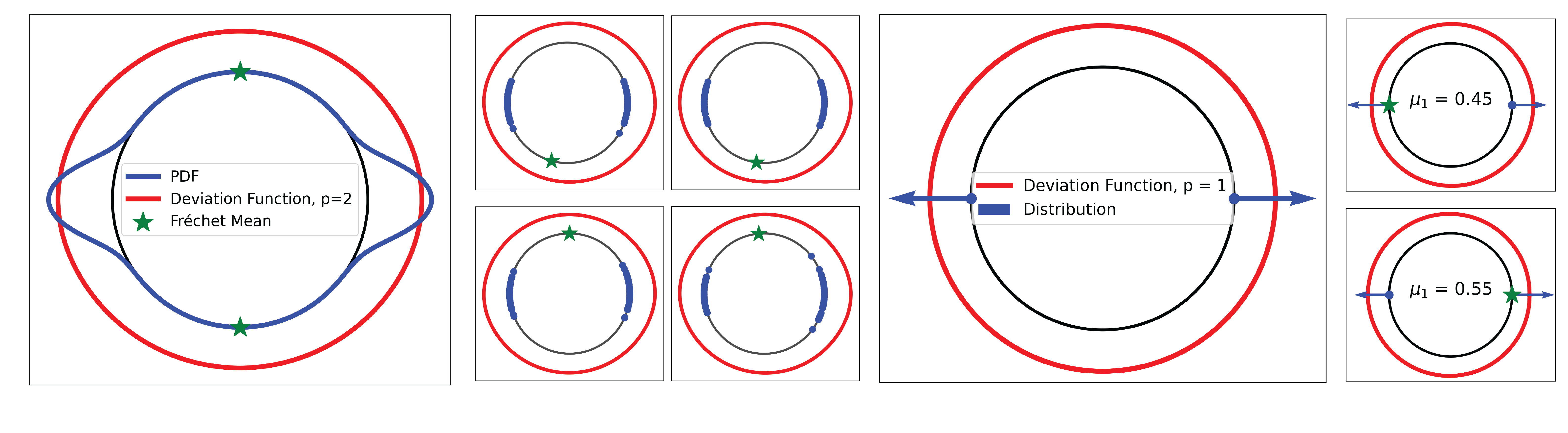} \put(22,0){(a)}
    \put(68,0){(b)}
    \put(113,0){(c)}
    \put(149,0){(d)}\end{overpic}
    \caption{\small Instability of global minima of $\ct_p$ on the circle. {\bf (a)} A pdf on the circle, its corresponding deviation function $\ct_2$, and the global minimizers of $\ct_2$. The deviation function and the pdf are both  rescaled in this plot for visual clarity. {\bf (b)} Four draws of 100 samples from the distribution shown in (a); $\ct_2$ is plotted for each empirical distribution, along with its global minimizer, which varies wildly between samples. {\bf (c)} A distribution on the circle consisting of equally weighted Dirac masses supported on antipodal points (represented by equal length arrows) and its deviation function $\ct_1$. Here, $\ct_1$ is constant, so every point on the circle is a minimizer (i.e., a median). {\bf (d)} Measures consisting of Dirac masses where the weight on the mass at $(1,0)$ is given by $\mu_1$, together with their associated deviation functions $\sigma_1$. If $\mu_1 > 1/2$, the unique median lies at $(1,0)$, while the median lies at $(-1,0)$ if $\mu_1 < 1/2$, indicated with a star in either case.}
    \label{fig:FrechetMeansCircle}
\end{figure}
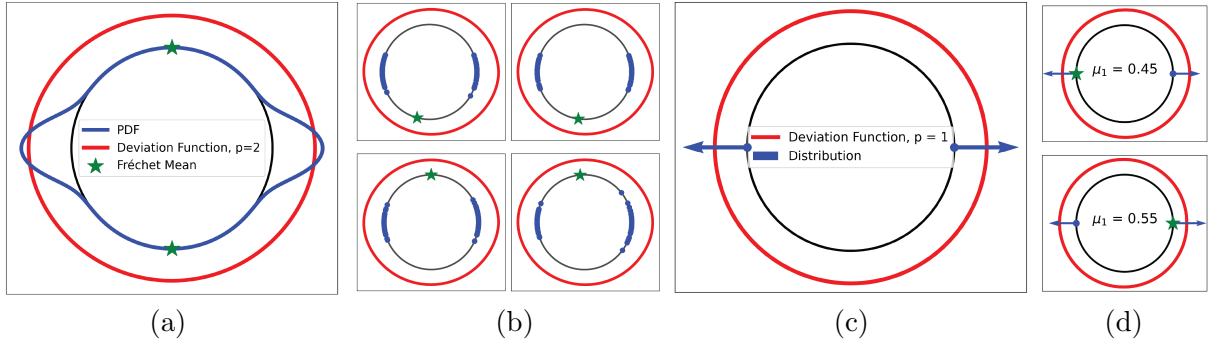

\begin{example}[Instability of Means]\label{E:circle}
This example illustrates the instability of global minima of $\ct_2$. Take the underlying metric space to be the unit circle endowed with geodesic distance and let $\mu$ be the distribution with two modes shown in Figure \ref{fig:FrechetMeansCircle}(a), which also displays the deviation function $\sigma_2$ of $\mu$ and the Fr\'{e}chet mean set formed by the north and south pole of the circle. We sample 100 points from $\mu$ and calculate the global Fr\'echet means of the associated empirical distributions. Figure \ref{fig:FrechetMeansCircle}(b) shows that the resulting Fr\'echet mean bounces between a point near the north pole and a point near the south pole, depending on the particular sample. This shows, qualitatively, that the Fr\'echet mean is unstable. Moreover, this same behavior persists as the number of samples is increased, as illustrated quantitatively in Example \ref{ex:quantitative}. 
\hfill $\Diamond$
\end{example}

\begin{example}[Instability of Medians]\label{E:circle_median}
To see that medians are unstable, we once again consider an example where the metric space is the unit circle, endowed with geodesic distance. Let $\mu$ be a measure consisting of a weighted sum of Dirac masses---one with weight $\mu_1$ at the point $(1,0)$ and another with weight $\mu_2 = 1- \mu_1$ at $(-1,0)$. Parameterizing the circle by angle $\varphi \in [0,2\pi)$, with $\varphi=0$ corresponding to $(1,0)$, it is not hard to show that $\ct_1$ is given explicitly by
\begin{equation}
\ct_1(\varphi) = \left\{\begin{array}{ll}
(\mu_1-\mu_2)\varphi + \mu_2 \pi & 0 \leq \varphi \leq \pi \\
(\mu_2 - \mu_1)\varphi + (2\mu_1 - \mu_2) \pi & \pi \leq \varphi < 2\pi.
\end{array}\right.
\end{equation}
If $\mu_1 = \mu_2 = 1/2$ then $\ct_1$ is constant, as is illustrated in Figure \ref{fig:FrechetMeansCircle}(c), so all points on the circle are medians of $\mu$. On the other hand, if $\mu_1 > 1/2$ then the unique minimizer of $\ct_1$ (i.e., the median) lies at $(1,0)$ and if $\mu_1 < 1/2$ then the unique minimizer lies at $(-1,0)$. Examples are shown in Figure \ref{fig:FrechetMeansCircle}(d).
\hfill $\Diamond$
\end{example}

\begin{figure}
    \centering
    \begin{overpic}[abs,unit=1mm,scale=.285]{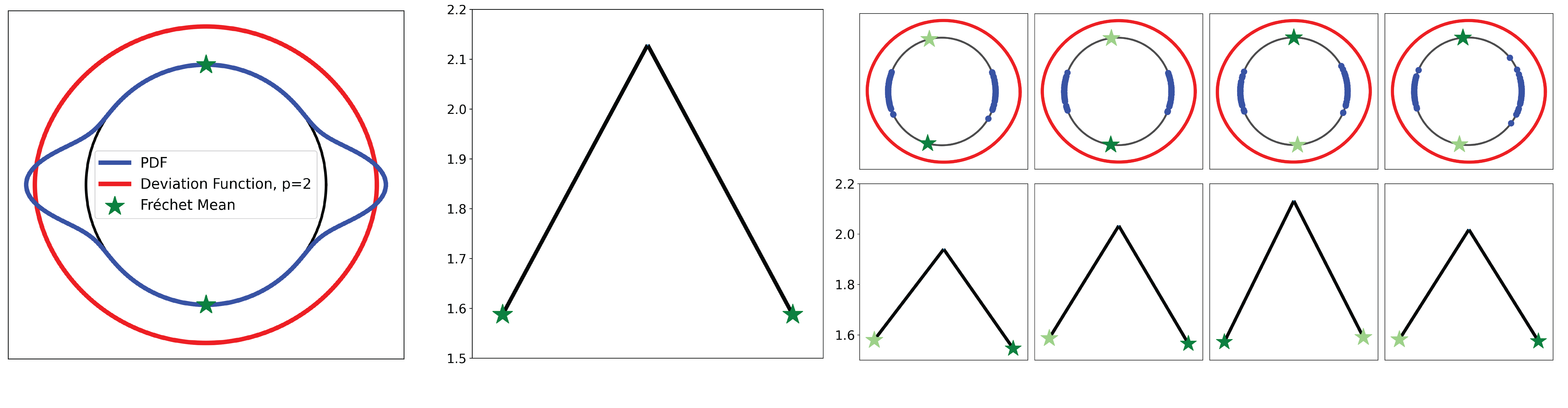} \put(22,0){(a)}
    \put(65,0){(b)}
    \put(124,0){(c)}
    \end{overpic}
    \caption{\small Stability of Barycentric Merge Trees on the circle. {\bf (a)} A circle with a bimodal distribution, its deviation function $\ct_2$ and the Fr\'echet mean set, as in Figure \ref{fig:FrechetMeansCircle}. {\bf (b)} The BMT for the distribution in (a); the leaves of the tree correspond to the global minima of $\ct_2$ (indicated by the stars), and the height of the merge point at the top corresponds to the maximum value of $\ct_2$. {\bf (c)} BMTs for empirical distributions drawn from the pdf. Leaves of the trees correspond to marked points on the circles. Observe that the merge trees all have similar structure, reflecting the stability of the BMTs as summaries of the behavior of the deviation function.}
    \label{fig:StabilityOfTrees}
\end{figure}

\begin{example}[Stability of Barycentric Merge Trees]
    We reconsider the distributions from Example \ref{E:circle} and Figures \ref{fig:FrechetMeansCircle}(a)(b). Figure \ref{fig:StabilityOfTrees} shows the barycentric merge trees associated to the distributions. One can intuitively see that the resulting structures share a common structure; this stability property is a reflection of the Lipschitz bound established in Theorem \ref{T:stab}. 
\hfill $\Diamond$
\end{example}

\subsection{Modes}

The representation and estimation of the {\em modes} of a probability distribution on a metric space $(X,d_X)$ are problems conceptually analogous to their counterparts for barycenters, as modes can be thought of as the most central points within pockets of probability mass. The shared centrality theme suggests a close connection between barycenters and modes, the main contrast being that, for modes,  centrality takes a more localized form. However, even in very simple cases as Example \ref{E:circle}, barycenters as defined above can lie in ``data deserts'', away from regions of concentration of mass and contrary to the very concept of modes. For this reason, we broaden the definition of Fr\'{e}chet variance and deviation functions to be able to frame the barycenter and mode problems as one. Instead of defining variance and deviation functions only with respect to the base metric $d_X$, in \eqref{E:mdistance}, we allow other (pseudo) metrics $\fdist \colon X \times X \to \real$. Then, $V_p$ takes the form 
\begin{equation}
V_p (x) = \mathbb{E}_\mu[\theta^p(x,\cdot)] = \int_X \theta^p(x,y) d\mu(y).
\end{equation}

\begin{wrapfigure}[14]{r}{0.45\textwidth} 
\vspace{-20pt}
  \begin{center}
    \includegraphics[width=0.43\textwidth]{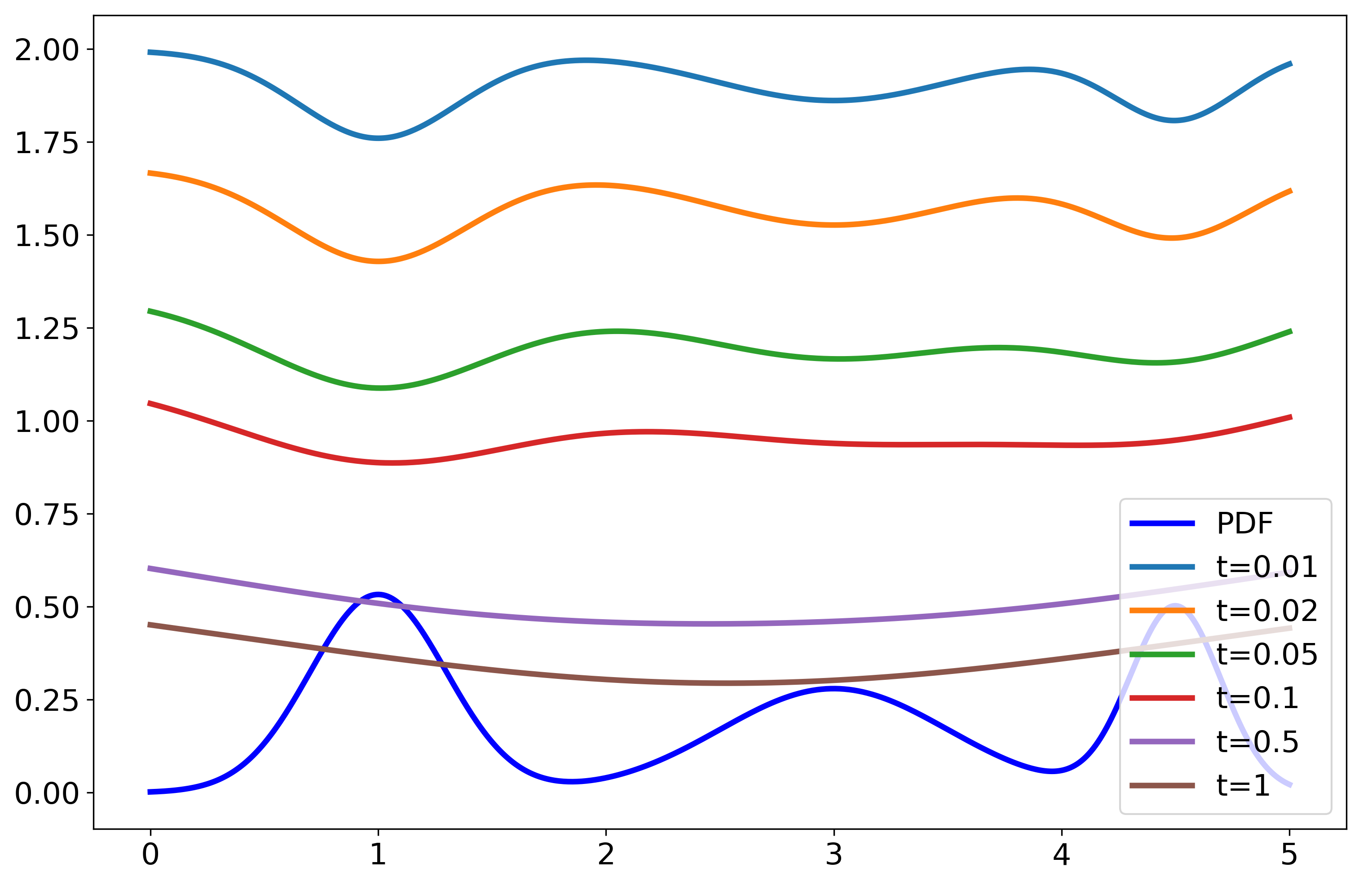}
    \caption{\small Fr\'echet variance functions associated with the heat kernel at different scales for a pdf on the real line.}
    \label{fig:RealLine}
  \end{center}
  \vspace{-10pt}
\end{wrapfigure}

As the notions of pockets of mass or clusters are inherently scale dependent, in mode detection, we employ diffusion distances $\fdist$ derived from kernel functions $k \colon X \times X \to \real^+$ \cite{coifman2006}, as detailed in Section \ref{S:modes}. The value $k(x,y)$ can be thought of as quantifying the level of communication between the points $x$ and $y$, so that if both points lie in the support of the measure, mass near $x$ and mass near $y$ get clustered together by the kernel if $k(x,y)$ is sufficiently large. This produces a chaining effect that places a sequence of sufficiently $\fdist$-close local pockets of mass all into the same cluster. This diffusion distance formulation allows us to approach modes as a barycenter problem. For Riemmanian manifolds with the geodesic distance, the heat kernel associated with the Laplace-Beltrami operator gives a natural one-parameter family of kernels $k_t$, $t>0$, to employ in mode analysis. We thus obtain a family of BMTs representing $(X,d_X,\mu)$, parameterized by $t>0$. In the tradition of scale space approaches to data analysis (cf.\,\cite{lindeberg1994,chaudmarron2000}), different structural properties are captured across the range of scales and the family can also reveal scales at which the barycentric merge trees undergo ``phase transitions''. Figure \ref{fig:RealLine} illustrates the connection between diffusion distances and modes of a distribution on the real line. The  variance function $V_2$ is calculated with respect to diffusion distances derived from the heat kernel for several values of $t>0$ (see Section \ref{S:modes} for details and also \cite{diaz2019probing} for additional examples). Observe that, for small values of $t$, the local minima of the Fr\'echet variance now correspond to local maxima (i.e., local modes) of the pdf. This illustrates how the framework proposed in this paper simultaneously applies to the barycenter and mode estimation problems.

\begin{example}[Modes and BMTs]
Figure \ref{fig:mode_merge_trees} shows the local mode detection pipeline applied to the circle distribution from Figure \ref{fig:BMT_example}. Here an exponential kernel, depending on a scale parameter $t > 0$, is applied to the geodesic distance, which in turn determines the diffusion distance. Merge trees for the associated deviation function $\ct_2$ are shown, together with the associated local modes on the circle (i.e., points corresponding to leaves on the trees). Increasing the scale parameter gives a coarser picture of the mode structure.
\hfill $\Diamond$
\end{example}

\begin{figure}
    \centering
    \includegraphics[width=\linewidth]{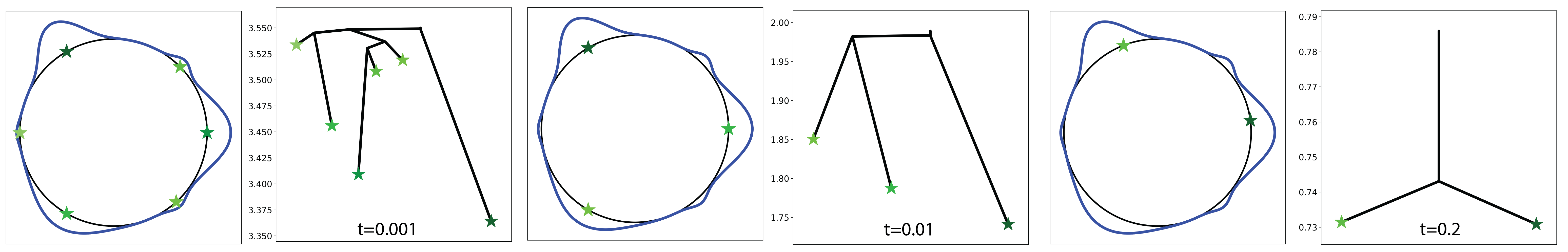}
    \caption{\small Merge trees for modes: each pair shows the distribution on the circle from Figure \ref{fig:BMT_example} with modes indicated by $\star$. The BMTs are calculated with respect to diffusion distances depending on an exponential kernel that has a scale parameter $t>0$. The results shown are for various values of this parameter.}
    \label{fig:mode_merge_trees}
\end{figure}

\subsection{Barycentric Merge Trees}

Continuous functions $f \colon X \to \real$ defined on connected domains can be succinctly represented by {\em merge trees} $T_f$, certain silhouettes of $f$ visualized as rooted trees, that have been studied primarily from a metric perspective (cf.\,\cite{morozov2013,beketayev2014,curry2022}). The present approach to barycentric merge trees differs from the purely metric view as it is based on a metric-measure formulation. To our knowledge, such a formulation has only been considered in a heuristic manner in the merge-tree literature, as in \cite{curry2022}. $T_f$ is the quotient space of $X$ under the equivalence relation $x \sim y$ if there exists $t \in \real$ such that $f(x)=f(y)=t$ and both $x$ and $y$ lie in the same connected component of the sublevel set $f^{-1}((-\infty,t])$. By construction, $f$ descends to a function $\hat{f} \colon T_f \to \real$. Using merge heights for pairs of points in $T_f$, as measured by $\hat{f}$-values, one can equip $T_f$ with a (pseudo) metric $d_f$ \cite[Definition 16]{curry2022} analogous to cophenetic distances for phylogenetic trees and dendograms \cite{sokal1962}. This turns a merge tree into a functional (pseudo) metric space $(T_f, d_f, \hat{f})$ and a functional variant of the Gromov-Hausdorff distance provides a means for quantifying similarities and contrasts between two merge trees $(T_f,d_f, \hat{f})$ and $(T_g,d_g, \hat{g})$, as done in \cite{curry2024} in the more general setting of Reeb graphs. A key difference between this type of metric and, say, the sup metric applied directly to $f$ and $g$, is the geometry captured in a more explicit and compact form. The approach to merge trees in \cite{morozov2013,beketayev2014} is different and based on interleaving distances.

The merge tree $T_p$ associated with a $p$-deviation function $\ct_p \colon X \to \real$ not only is equipped with a metric structure $d_p$ and a function $\hat{\sigma}_p \colon T_p \to \real$, as described above, but also with a probability measure $\mu_p$, namely, the pushforward of the original distributions $\mu$ on $X$ to the quotient space $T_p$. We thus have a representation of the original $mm$-space $\mm{X} = (X,d_X,\mu)$ by a summary (functional) $mm$-space $(T_p,d_p,\mu_p, \hat{\sigma}_p)$. This additional probabilistic structure enables us to analyze variation in barycentric merge trees in a \emph{(functional) Gromov-Wasserstein framework}~\cite{memoli2011,vayer2020fused,anbouhi2024}, instead of taking a purely metric approach. Our main stability result, obtained in Theorem \ref{T:stab}, asserts that if $\mu$ and $\mu'$ are probability measures on $X$, then the distance between their barycentric merge trees, viewed as functional $mm$-spaces, admits an upper bound of the order of the Wasserstein distance between $\mu$ and $\mu'$. Corollary \ref{C:convergence} records the fact that this form of stability implies consistency, and we also obtain estimates for the rate of convergence of empirical barycentric merge trees to their theoretical population model. These rates are derived directly from the Stability Theorem and results by Weed and Bach on rates for Wasserstein convergence of empirical measures \cite{weed2019}. Stability and consistency for modes are included in these results because we frame modes as special cases of barycenters.

As the deviation functions for empirical measures are defined on the entire space $X$, we also propose a discretization scheme for empirical barycentric merge trees, which are the BMTs of practical interest. For $X$ compact, starting from a finite $\delta$-fine grid $V \subseteq X$, $\delta>0$, we construct a combinatorial barycentric merge tree whose vertex set support a (functional) metric measure structure that is $\delta$-close to the BMT of $\mu_n$ with respect to a Gromov-Wasserstein type distance. Since the combinatorial trees so obtained can be rather complex, with as many vertices as the cardinality of $V$, we also discuss a simplification method that gives a provably accurate approximation.

\subsection{Organization}

Section \ref{S:prelim} defines Fr\'{e}chet variance functions and $p$-deviation functions in a form that allows for a joint treatment of barycenters and modes, and also discusses connectivity properties on $X$ that are needed to obtain stable barycentric merge trees. Section \ref{S:bmt} constructs barycentric merge trees as (pseudo) metric measure spaces and examines some properties they satisfy. Section \ref{S:stabcon} proves the main results of the paper, namely, stability and consistency of barycentric merge trees. The application to analysis of modes is developed in Section \ref{S:modes}, and Section \ref{S:discrete} is devoted to discrete models. Section \ref{S:discussion} closes the paper with a summary and some discussion.


\subsection*{Acknowledgements}

The authors thank Nelson A. Silva for carefully reading the manuscript and suggesting several improvements to the text. This research was partially done while WM was visiting the Institute for Mathematical Sciences, National University of Singapore in 2024. TN was partially supported by NSF grants DMS--2107808 and DMS--2324962.


\section{Preliminaries} \label{S:prelim}

In this paper, unless otherwise stated, $(X,d_X)$ is a connected and locally path-connected Polish metric space (Polish meaning complete and separable). For spaces satisfying these connectivity hypotheses, we repeatedly use the fact that any open, connected subset $U \subseteq X$ is path connected. To simplify notation, for $w, z \in \real$, we adopt the abbreviations $w\vee z = \max\{w,z\}$ and $w \wedge z = \min\{w,z\}$ throughout the text. 


\subsection{Fr\'{e}chet Variance and Deviation Functions} \label{S:ffunction}

As explained in the Introduction, we take an approach to detection and estimation of modes of a probability distribution on a metric space $(X,d_X)$ that requires versions of the Fr\'{e}chet variance function based on distance functions other than $d_X$. For this reason, we define Fr\'{e}chet variance with respect to more general pseudo metrics $\fdist \colon X \times X \to \real$, as this lets us frame modes as special cases of barycenters. A discussion of pseudo metrics derived from kernel functions and well suited to mode analysis is presented in Section \ref{S:modes}.

\begin{definition} \label{D:fvariance}
Let $\mm{X} = (X,d_X,\mu)$ be an $mm$-space with $\mu \in \borel{X,d_X;p}$, $p \geq 1$.    
\begin{enumerate}[\rm (i)]
\item A pseudo-metric $\fdist \colon X \times X \to \real$ is {\em $L$-admissible}, $L>0$, if $\fdist(x,y) \leq L d_X (x,y)$, $\forall x,y \in X$. We refer to such $L$ as an {\em admissibility constant} for $\fdist$. We say that $\fdist$ is \emph{admissible} if it is $L$-admissible for some $L > 0$.
\item The {\em Fr\'{e}chet $p$-variance function} $V_p \colon X \to \real$ associated with an admissible $\fdist$ is defined by
\[
V_p (x)\coloneqq \int_X \fdist^p(x,y) \, d\mu(y).
\]
The fact that $\mu$ has finite $p$-moments ensures that $V_p$ is well defined. We omit $\fdist$ from the Fr\'{e}chet function notation to keep it simple, as the choice of $\fdist$ should always be clear from the context.
\item The {\em $p$-deviation function} $\ct_p \colon X \to \real$ is defined as $\ct_p (x)\coloneqq \big(V_p(x)\big)^{1/p}$.
\end{enumerate}
\end{definition}

The standard Fr\'{e}chet variance function is that associated with $\fdist=d_X$ and $p=2$. The following proposition describes a basic property of deviation functions needed in our stability arguments.

\begin{proposition} \label{P:lip}
If $\fdist$ is an admissible pseudo metric and $\mu \in \borel{X,d_X;p}$, then $|\ct_p(x)-\ct_p(y)| \leq\fdist(x,y)$, for any $x,y \in X$.
\end{proposition}
\begin{proof}
For $x,y \in X$, by the Minkowski inequality we have that
\begin{equation}
\begin{split}
|\ct_p (x) - \ct_p (y)| &= 
\Big| \Big( \int_X \fdist^p (x,z) \, d\mu (z) \Big)^{1/p} - \Big( \int_X \fdist^p (y,z) \, d\mu (z) \Big)^{1/p} \Big| \\
&\leq \Big( \int_X |\fdist(x,z) - \fdist(y,z)|^p \, d\mu (z) \Big)^{1/p} \leq \Big( \int_X  \fdist^p(x,y) \, d\mu (z) \Big)^{1/p} = \fdist(x,y).
\end{split}
\end{equation}
\end{proof}


\subsection{Connectivity Modulus}

This section introduces the notions of {\em connectivity constant} and {\em connectivity modulus} for an admissible pseudo metric $\fdist \colon X \times X \to \real$ on $(X,d_X)$. Let $I=[0,1]$. Given $x,y \in X$, let $\Gamma(x,y)$ denote the collection of all (continuous) paths $\gamma \colon I \to (X,d_X)$ such that $\gamma(0)=x$ and $\gamma(1)=y$. Define the {\em merge distance function} $r_\fdist \colon X \times X \to \real$ by
\begin{equation} \label{E:mergec}
r_\fdist (x,y) \coloneqq \inf_{\gamma \in \Gamma(x,y)} \sup_{t\in I} \,\fdist(x,\gamma(t)) \vee \fdist(\gamma(t),y).
\end{equation}
For $\theta=d_X$, we use the notation $r_X$ for the merge distance function. Clearly, 
\begin{equation}
\sup_{t\in I} \fdist(x,\gamma(t)) \vee \fdist(\gamma(t),y) \geq \fdist(x,\gamma(0)) \vee \fdist (\gamma(0),y) = \fdist(x,y),
\end{equation}
so that $r_\fdist (x,y) \geq \fdist (x,y)$, for any $x,y \in X$. In particular, $r_X \geq d_X$. By reversing path orientation, one readily verifies that $r_\fdist$ is a symmetric function. Intuitively, $r_\theta(x,y)$ quantifies how far out, as measured by $\fdist$, we have to go from $x$ and $y$ to be able to connect $x$ and $y$ by a path.

\begin{definition} \label{D:modulus}
Let $(X,d_X)$ be a metric space, $\fdist \colon X \times X \to \real$ an admissible pseudo metric, and $K>0$.
\begin{enumerate}[\rm (i)]
 \item $K$ is a {\em $\fdist$-connectivity constant} for $(X,d_X)$ if $r_\fdist(x,y) \leq K d_X(x,y)$, $\forall x,y \in X$. If $\fdist=d_X$, then $K \geq 1$ and we refer to $K$ simply as a {\em connectivity constant} for $(X,d_X)$.
 \item The {\em $\fdist$-connectivity modulus} $K_\fdist$ of $(X,d_X)$ is defined as
 \[
 K_\fdist\coloneqq \sup_{\substack{x,y \in X \\ x\ne y}} \frac{r_\fdist (x,y)}{d_X (x,y)}\,.
 \]
 For $\fdist=d_X$, we write $K_\fdist = K_X$ and call $K_X$ the {\em connectivity modulus} of $(X,d_X)$.
 \end{enumerate}
\end{definition}

If the $\fdist$-connectivity modulus is finite, $K_\fdist$ is the smallest $\fdist$-connectivity constant for $(X,d_X)$.

\begin{proposition} \label{P:geodesic}
If $(X, d_X)$ is a geodesic space, then $K_X = 1$.
\end{proposition}

\begin{proof}
We show that $r_X = d_X$, which implies that $K_X=1$. As noted above, $r_X \geq d_X$. For the opposite inequality, let $\gamma_{x,y} \colon I \to X$ be a geodesic from $x$ to $y$. Then,
\begin{equation}
\begin{split}
d_X (x,\gamma_{x,y}(t)) \vee d_X (\gamma_{x,y}(t), y) &\leq d_X (x,\gamma_{x,y}(t)) + d_X (\gamma_{x,y}(t), y) = d_X (x,y) \,,
\end{split}
\end{equation}
$\forall t \in I$. Taking the infimum over $\gamma \in \Gamma(x,y)$, we obtain $r_X (x,y) \leq d_X (x,y)$.
\end{proof}

\begin{proposition} \label{P:kmodulus}
Let $\fdist \colon X \times X \to \real$ be an $L$-admissible pseudo metric. If $K\geq 1$ is a connectivity constant for $(X,d_X)$, then $KL$  is a $\fdist$-connectivity constant for $(X,d_X)$; that is,
$r_\fdist (x,y) \leq KL d_X(x,y)$.
\end{proposition}

\begin{proof}
By assumption, for any $x,y \in X$ and $\epsilon >0$, there exists a path $\gamma \in \Gamma(x,y)$ such that $\sup_{t\in I} d_X(x, \gamma(t)) \vee d_X(\gamma(t),y) \leq \epsilon + K d_X(x,y)$. Then,
\begin{equation}
\sup_{t\in I} \fdist(x, \gamma(t)) \vee \fdist(\gamma(t),y) \leq L \sup_{t\in I} \big(d_X(x, \gamma(t)) \vee d_X(\gamma(t),y)\big) \leq L \epsilon + LK d_X(x,y) .
\end{equation}
Since $\epsilon>0$ is arbitrary, we have $\sup_{t\in I} \fdist(x, \gamma(t)) \vee \fdist (\gamma(t),y) \leq KL d_X(x,y)$, which in turn implies that $r_\fdist (x,y) \leq KL d_X(x,y)$.
\end{proof}


\section{Barycentric Merge Trees} \label{S:bmt}

This section introduces our main objects of study, barycentric merge trees. For $t \in \real$, we adopt the notation $A(t)\coloneqq \ct_p^{-1} ((-\infty, t]) = \ct_p^{-1} ([0, t])$ for the sublevel sets of $\ct_p$. The last equality holds because $\ct_p \geq 0$. 

\begin{definition} \label{D:ftree}
The {\em barycentric merge tree (BMT)} of $\mm{X}$ of order $p$ for an admissible pseudo metric $\fdist$, $p \geq 1$, is denoted $\tree{X}$ and defined as the quotient space of $X$ under the equivalence relation $x \sim y$ if there exists $t \in \real$ such that $\ct_p (x) = \ct_p (y) =t$ and $x$ and $y$ lie in the same connected component of $A(t)$. The quotient map is denoted $\alpha_p \colon X \to \tree{X}$.
\end{definition}

Points in $\tree{X}$ are in one-to-one correspondence with the connected components of the sublevel sets of $\ct_p$, as follows. For $x \in X$, let $t = \ct_p(x)$ and $\alpha_p(x)=a \in \tree{X}$. Then, the point $a$ represents the connected component $C_a$ of $A(t)$ containing $x$. Clearly, $C_a$ only depends on the equivalence class of $x$. 

\begin{figure}
    \centering
    \includegraphics[width=0.9\linewidth]{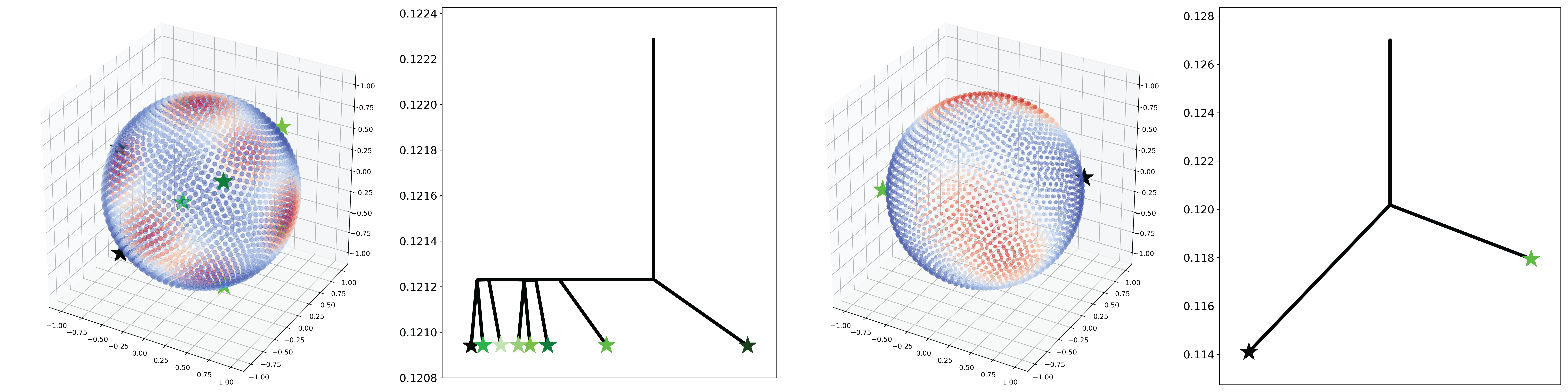}
    \caption{\small BMTs for distributions on the sphere (see Example \ref{ex:sphere}). {\bf Left:} Heat map representation of the pdf of a  highly symmetric distribution on the 2-sphere, with modes at each intersection of the sphere with the coordinate axes. The BMT for $p=2$ is shown, together with points corresponding to leaves on the sphere. {\bf Right:} A less symmetric distribution on the 2-sphere with its BMT and points corresponding to the leaves.}
    \label{fig:SphereTrees}
\end{figure}

\begin{example}[BMTs on the Sphere]\label{ex:sphere}
    We illustrate the barycentric merge tree construction for distributions on the 2-sphere in Figure \ref{fig:SphereTrees}. In each case, we use the standard geodesic distance as the metric, and consider the associated deviation function $\sigma_2$. The first example is for a highly symmetric distribution, given by a pdf which exponentially decays in distance from the set of points $\{(\pm 1,0,0),(0,\pm 1, 0), (0,0,\pm 1)\}$. The resulting BMT has leaves corresponding to the faces of the standard cross-polytope, reflecting the symmetry of the measure. These points are the local minima of $\sigma_2$. The leaves merge at a common height---corresponding to the unique local maximum value of the pdf---after which the sublevel sets of $\sigma_2$ are connected. The second example is a random transformation of the first, with no obvious symmetries, and results in a BMT with two leaves. Observe that in both examples, the computed local Fr\'echet means fall in low density areas, motivating the mode detection problem that we consider in Section \ref{S:modes}.
    
    All numerical examples in the paper are implemented via discrete approximations of the underlying spaces, following the method presented in Section \ref{S:discrete}. \hfill $\Diamond$
\end{example}

\smallskip
By construction, the deviation function $\ct_p \colon X \to \real$ descends to a function $\qct_p \colon \tree{X} \to \real$ such that the diagram
\begin{equation} \label{E:functions}
\begin{tikzcd}
X \arrow[dr,"\ct_p"'] \arrow[rr, "\alpha_p"] & & \tree{X}  \arrow[dl, "\qct_p"] \\
& \real &
\end{tikzcd}
\end{equation}
commutes (this is the function which was denoted as $\hat{\sigma}_p$ in the introduction). $\tree{X}$ can be equipped with a poset (partially ordered set) structure that we now describe.

\begin{definition} \label{D:poset}
For $a \in \tree{X}$, let $C_a \subseteq X$ be the connected component of the sublevel set of $\ct_p$ represented by $a$. 
\begin{enumerate}[(i)]
\item The partial order $\preceq$  is defined by $a \preceq b$ if $C_a \subseteq C_b$.
\item A point $a \in T_p (\mm{X})$ is a {\em leaf} if it is a minimum point of the poset; that is, if $b \preceq a$, then $a=b$.
\item Given $a,b,c \in \tree{X}$, $c$ is a {\em merge point} for $a$ and $b$ if $a \preceq c$ and $b \preceq c$.
\end{enumerate}
\end{definition}

The partial ordering $\preceq$ is compatible with the quotient topology on $\tree{X}$ in the sense that the graph of $\preceq$, given by 
\begin{equation}
\{(a,b) \in \tree{X} \times \tree{X} \colon a \preceq b\} ,
\end{equation}
is a closed subspace of $\tree{X} \times \tree{X}$. Thus, $\tree{X}$ is a {\em pospace} (partially ordered space), not just a poset. 

\begin{remark} \label{R:lowerset}
The following properties satisfied by $\preceq$ are readily verified:
\begin{enumerate}[(i)]
\item The function $\qct_p \colon \tree{X} \to \real$ is monotonic; that is, $\qct_p (a) \leq \qct_p (b)$ if $a \preceq b$;
\item For $a \in T_p(\mm{X})$, $C_a = \alpha_p^{-1} (\lset{a})$, where $\lset{a} \, \coloneqq\{b \in T_p \mm{X}) \colon b \preceq a\}$ is the {\em principal lower set} of $a$;
\item If $a \preceq b$ and $a \preceq c$, then $b\preceq c$ or $c \preceq b$; that is, $\tree{X}$ has no branching points;
\item If $a \preceq b$ and $\qct_p (a) = \qct_p (b)$, then $a =b$.
\end{enumerate}
\end{remark}

We now turn $\tree{X}$ into a (pseudo) $mm$-space starting with a (pseudo) metric on $\tree{X}$. For $a,b \in \tree{X}$, let 
\begin{equation}
\mpt(a,b) = \{c \in \tree{X} \colon a \preceq c \ \text{and} \ b \preceq c\}
\end{equation}
be the {\em merge set} of $a$ and $b$. Note that $\mpt(a,b) \ne \emptyset$. Indeed, let $x_a, x_b \in X$ be such that $\alpha_p(x_a)=a$ and $\alpha_p(x_b) = b$, and let $\gamma \colon I \to X$ be a path from $x_a$ to $x_b$. Such $\gamma$ exists because $X$ is path connected. Set $t_0 = \text{argmax}\, \ct_p(\gamma(t))$, $x_0 = \gamma(t_0)$, and $c =\alpha_p (x_0)$. Then, $c \in \mpt (a,b)$.

\begin{definition}[cf.\,\cite{curry2022,gasparovic2025intrinsic}] \label{D:treedist}
The distance function $\pdist{X} \colon \tree{X} \times \tree{X} \to \real$ is defined as
\[
\begin{split}
\pdist{X} (a,b) &\coloneqq \inf_{c \in \mpt(a,b)} \big(\qct_p (c) - \qct_p (a)\big) \vee \big(\qct_p (c) - \qct_p(b)\big) \\
&\,\, = \inf_{c \in \mpt(a,b)} \qct_p (c) - (\qct_p (a) \wedge \qct_p (b)).
\end{split}
\]
\end{definition}

The next goal is to give an alternative description of $\pdist{X}$ in terms of variation of $\ct_p$ along paths in $X$ and show that $\pdist{X}$ is indeed a pseudo metric. For $\gamma \in \Gamma(x,y)$, let
\begin{equation} \label{D:predistance}
\begin{split}
\rho(\gamma) &= \sup_{t \in I} \big(\ct_p (\gamma(t)) -\ct_p (x)\big) \vee \big(\ct_p (\gamma(t)) -\ct_p (y)\big) \\
&= \sup_{t \in I} \ct_p (\gamma(t)) - (\ct_p (x) \wedge \ct_p (y)) \,.
\end{split}
\end{equation}

\begin{remark} \label{R:trace}
$\rho (\gamma)$ only depends on the trace of $\gamma$, as it is invariant under reparameterizations of the curve in a very flexible sense. Namely, for any mapping $h \colon I \to I$ such that $h(0)=0$ and $h(1)=1$, the curve $\gamma' = \gamma \circ h$ satisfies $\rho (\gamma) = \rho (\gamma')$.
\end{remark}

\begin{definition}
For $\mu \in \borel{X,d_X;p}$, the {\em merge radius function} $\mdist{X} \colon X \times X \to \real$ of $\mm{X}= (X,d_X,\mu)$  is defined as
\[
\begin{split}
\mdist{X} (x,y) \coloneqq \inf_{\gamma \in \Gamma(x,y)} \rho (\gamma) 
= \inf_{\gamma \in \Gamma(x,y)} \sup_{t \in I} \ct_p (\gamma(t)) - (\ct_p (x) \wedge \ct_p (y)).
\end{split}
\]
\end{definition}

\begin{proposition} \label{P:mradius}
Let $\mm{X} = (X, d_X, \mu)$ be an $mm$-space with $\mu \in \borel{X,d_X;p}$, $p \geq 1$. Then, for any $x,x',y,y',z \in X$, the following statements hold:
\begin{enumerate}[{\rm(i)}]
\item $\mdist{X} (x,z) \leq \mdist{X} (x,y) + \mdist{X} (y,z)$;
\item if $\alpha_p(x) = \alpha_p(x')$, then $\mdist{X} (x,x') =0$;
\item if $\alpha_p(x) = \alpha_p(x')$ and $\alpha_p(y)=\alpha_p(y')$, then $\mdist{X} (x,y) = \mdist{X} (x',y')$.
\end{enumerate}
\end{proposition}

\begin{proof}
The argument is similar to that given in \cite[Proposition 4]{curry2024} for a related result for Reeb spaces. 

\smallskip

(i) Given $\gamma_1 \in \Gamma(x,y)$ and $\gamma_2 \in \Gamma(y,z)$, let $\gamma \in \Gamma(x,z)$ be the concatenation of $\gamma_1$ and $\gamma_2$. Then,
\begin{equation}
\begin{split}
\sup_{t\in I} \ct_p (\gamma (t)) - \ct_p(x) &= \big[ \sup_{t\in I} \ct_p (\gamma_1 (t)) - \ct_p(x)\big] \vee  \big[\sup_{t\in I} \ct_p (\gamma_2 (t)) - \ct_p(x)\big] \\
&= \big[ \sup_{t\in I} \ct_p (\gamma_1(t)) - \ct_p(x)\big] \vee  \big[\sup_{t\in I} \ct_p (\gamma_2(t)) - \ct_p(y) + \ct_p(y) - \ct_p(x) \big] \\
&\leq \big[ \sup_{t\in I} \ct_p (\gamma_1 (t)) - \ct_p(x)\big] +  \big[\big(\sup_{t\in I} \ct_p (\gamma_2 (t)) - \ct_p(y)\big) \big] \leq \rho(\gamma_1) + \rho(\gamma_2) \,,
\end{split}
\end{equation}
for, $\ct_p(y) - \ct_p (x) \leq \sup_{t\in I} \ct_p (\gamma_1 (t))) - \ct_p(x)$. Similarly, $\sup_{t\in I} \ct_p (\gamma (t)) - \ct_p(z) \leq \rho(\gamma_1) + \rho(\gamma_2)$. Therefore, $\rho(\gamma) \leq \rho(\gamma_1) + \rho(\gamma_2)$, which implies that
\begin{equation}
\mdist{X}(x,z) =  \inf_{\gamma \in \Gamma(x,z)} \rho(\gamma) \leq \inf_{\gamma_1 \in \Gamma(x,y)} \rho(\gamma_1) + \inf_{\gamma_2 \in \Gamma(y,z)} \rho(\gamma_2) = \mdist{X} (x,y) + \mdist{X} (y,z)\,.
\end{equation}

(ii) Let $t=\ct_p(x) = \ct_p(x')$ so that $x$ and $x'$ fall in the same connected component of the sublevel set $\ct_p^{-1} ((-\infty,t])$. For any $\epsilon>0$, $x$ and $x'$ must fall in the same connected component $C_\epsilon$ of the open set $\ct_p^{-1} ((-\infty,t + \epsilon))$. Therefore, there is a path $\gamma \in \Gamma(x,y)$ whose image is contained in $C_\epsilon$. This implies that $\mdist{X}(x,x') \leq \rho (\gamma) < \epsilon$. Since $\epsilon>0$ is arbitrary, $\mdist{X}(x,x')=0$.

\smallskip

(iii) By (i) and (ii),
\begin{equation}
\mdist{X} (x,y) \leq \mdist{X} (x,x') + \mdist{X} (x',y') + \mdist{X} (y',y) = \mdist{X} (x',y').
\end{equation}
Similarly, $\mdist{X} (x',y') \leq \mdist{X} (x,y)$. This concludes the proof.
\end{proof}

Proposition \ref{P:mradius} implies that $\mdist{X}$ induces a pseudo metric on $\tree{X}$; that is, the function $\bardist{X} \colon \tree{X} \times \tree{X} \to \real$ given by
\begin{equation} \label{E:bardist}
\bardist{X} (\alpha_p(x),\alpha_p(y)) \coloneqq \mdist{X} (x,y)
\end{equation}
is a well-defined pseudo metric.

\begin{proposition} \label{P:distance1}
If $\mu \in \borel{X,d_X;p}$, $p \geq 1$, then $\pdist{X} = \bardist{X}$. In particular, $\dist{X}$ is a pseudo metric.
\end{proposition}

\begin{proof}
We first show that $\pdist{X} \leq \bardist{X}$. Let $a,b \in \tree{X}$ and $x_a, x_b \in X$ be such that $\alpha_p (x_a)=a$ and $\alpha_p (x_b)=b$. Given $\epsilon >0$, there is $\gamma \in \Gamma(x_a,x_b)$ such that $\sup_{t \in I} \ct_p (\gamma(t)) - (\ct_p (x) \wedge \ct_p (y)) < \bardist{X} (a,b) + \epsilon$. Let $t_0 \in I$ be a point that realizes this supremum and set $c = \alpha_p (\gamma(t_0))$. Clearly, $\ct_p (\gamma(t_0)) \geq \ct_p(x) \vee \ct_p (y)$ so that $c \in \mpt(a,b)$ and 
\begin{equation}
\pdist{X} (a,b) \leq \qct_p (c) - (\qct_p (a) \wedge \qct_p (b)) 
= \ct_p (\gamma(t_0)) - (\ct_p (x) \wedge \ct_p (y)) <  \bardist{X} (a,b) + \epsilon. 
\end{equation}
Since $\epsilon>0$ is arbitrary, we obtain the desired inequality.

For the opposite inequality, given $\epsilon >0$, let $c \in \mpt(a,b)$ be such that $\qct_p (c) - (\qct_p (a) \wedge \qct_p (b)) < \pdist{X} (a,b) + \epsilon$. Let $x_a, x_b, x_c \in X$ be representatives of the equivalence classes $a,b,c \in \tree{X}$, respectively, and $t_c = \qct_p (c)$. Then,
$x_a, x_b, x_c$ are contained in the same connected component $C$ of the open set $\ct_p^{-1} (-\infty, \pdist{X}(a,b) + \epsilon)$. Since $C$ is path connected, there is $\gamma \in \Gamma(x_a,x_b)$ whose image lies in $C$. This implies that $\bardist{X} (a,b) \leq \sup_{t \in I} \ct_p (\gamma(t)) - \ct_p (x_a) \wedge \ct_p (x_b) <  \pdist{X} (a,b) + \epsilon$. Since $\epsilon>0$ is arbitrary, we get $\bardist{X} (a,b) \leq \pdist{X} (a,b)$. 
\end{proof}

One can show that if $X$ is compact and $\tree{X}$ has finitely many leaves, then $\pdist{X}$ metrizes the quotient topology on $\tree{X}$ (cf.\,\cite{curry2022}). This, however, may not be true in general. Nonetheless, we show that $\pdist{X}$ has properties that make it well suited for our purposes: (a) the topology $\tau_p$ induced by $\pdist{X}$ is fine enough for the mapping $\qct_p \colon \tree{X}\to \real$ to be continuous (1-Lipschitz, as a matter of fact); (b) $\tau_p$ is coarse enough for $\alpha_p \colon X \to \tree{X}$ to be continuous; and (c) $\pdist{X}$ allows us to analyze BMT structural variation in a Gromov-Wasserstein type framework. As such, from this point on, we always assume that $\tree{X}$ is equipped with the metric $\pdist{X}$. 

\begin{proposition} \label{P:mcontinuous}
If $\mu \in \borel{X,d_X;p}$, then the following statements hold:
\begin{enumerate}[\rm(i)]
\item the map $\qct_p \colon (\tree{X}, \pdist{X}) \to \real$ is 1-Lipschitz;
\item the map $\alpha_p \colon (X,d_X) \to (\tree{X}, \pdist{X})$ is continuous.
\end{enumerate}
\end{proposition}

\begin{proof}
(i) Given $a,b \in \tree{X}$ and $\epsilon>0$, Definition \ref{D:treedist} ensures that there exists $c \in \mpt(a,b)$ such that 
\begin{equation}
(\qct_p (c) - \qct_p (a)) \vee (\qct_p (c) - \qct_p(b)) < \pdist{X}(a,b) + \epsilon \,.
\end{equation}
Since $c$ is a merge point for $a$ and $b$, we also have that $\qct_p(a) \leq \qct_p (c)$ and  $\qct_p(b) \leq \qct_p (c)$. Therefore, $|\qct_p (a) - \qct_p(b)| < \pdist{X} (a,b) +\epsilon$. Since $\epsilon>0$ is arbitrary, $|\qct_p (a) - \qct_p(b)| \leq \pdist{X} (a,b)$.

(ii) Let $U \subseteq \tree{X}$ be an open set and $\widetilde{U}= \alpha_p^{-1} (U) \subseteq X$. To show that $\widetilde{U}$ is open, given $x \in \widetilde{U}$ we construct an open neighborhood $V$ of $x$ such that $\alpha_p (V) \subseteq U$. Let $a = \alpha_p (x)$ and pick $\epsilon>0$ such that $B(a,\epsilon) \subseteq U$, where $B(a,\epsilon)$ is the open ball of radius $\epsilon$ centered at $a$. Since $X$ is locally path connected, there exists a path connected open neighborhood $V$ of $x$ with the property that $V \subseteq B(x,\epsilon/2L)$, where $L>0$ is an admissibility constant for $\fdist$. We claim that $\alpha_p (V) \subseteq B(a,\epsilon) \subseteq U$. Indeed, given any $y \in V$, let $\beta \in \Gamma(x,y)$ be a path contained in $V$ and thus in $B(x,\epsilon/2L)$. Proposition \ref{P:lip} ensures that $\ct_p$ is $L$-Lipschitz so that 
\begin{equation}
\begin{split}
|\ct_p (\beta(s)) - \ct_p (\beta(t))| &\leq |\ct_p (\beta(s)) - \ct_p (x)|+ |\ct_p (\beta(t)) - \ct_p (x)| \\
&\leq L d_X (\beta(s),x) + L d_X (\beta(t),x) < \epsilon, 
\end{split}
\end{equation}
for any $s,t \in I$. In particular, $\rho(\beta) = \sup_{t \in I} \big(\ct_p (\beta(t)) -\ct_p (x)\big) \vee \big(\ct_p (\beta(t)) -\ct_p (y)\big) < \epsilon$. This implies that
\begin{equation}
\bardist{X} (a, \alpha_p (y)) = \bardist{X} (\alpha_p (x), \alpha_p (y)) = \inf_{\gamma \in } \rho(\gamma) \leq \rho(\beta) < \epsilon \,.
\end{equation}
By Proposition \ref{P:distance1}, $\pdist{X} (a,\alpha_p(y)) < \epsilon$, showing that $\alpha_p (y) \in B(a,\epsilon)$. This concludes the proof.
\end{proof}


\begin{definition}
Let $\mu \in \borel{X,d_X,p}$ and $\mm{X} = (X, d_X, \mu)$. The {\em metric-measure barycentric merge tree} of order $p$ of $\mm{X}$ is the triple $(\tree{X}, \pdist{X}, \mu_p)$, where $\mu_p = \alpha_{p\sharp} (\mu)$ is the pushforward of $\mu$ under $\alpha_p \colon X \to \tree{X}$.
\end{definition}

Henceforth, we always assume this $mm$-structure on $\tree{X}$ and simplify the terminology to \emph{barycentric merge tree (BMT) of $\mm{X}$} if the choice of $p$ (and the admissible pseudo metric $\fdist$) is clear from the context.  We close this section with additional discussion of the metric properties of the map $\alpha_p$ under the assumption that $(X,d_X)$ has finite connectivity modulus (see Definition \ref{D:modulus}).

\begin{proposition} \label{P:qlip}
If $K>0$ is a connectivity constant for $(X, d_X)$ and $\fdist$ is $L$-admissible, then the quotient map $\alpha_p \colon (X,d_X) \to (\tree{X}, \pdist{X})$ is $KL$-Lipschitz.
\end{proposition}

\begin{proof}
Given $x,y \in X$ and $\epsilon>0$, by Proposition \ref{P:kmodulus} there is a path $\gamma_\epsilon \in \Gamma(x,y)$ such that 
\begin{equation} \label{E:thetabound}
\sup_{t \in I} \fdist (\gamma_\epsilon(t),x) \vee \fdist (\gamma_\epsilon(t),y) < \epsilon + r_\fdist (x,y)
\leq \epsilon + KL d_X (x,y),
\end{equation}
with $r_\fdist (x,y)$ as in \eqref{E:mergec}. Proposition \ref{P:lip} and \eqref{E:thetabound} imply that $\rho(\gamma_\epsilon) < \epsilon + KL d_X(x,y)$. Therefore,
\begin{equation}
\pdist{X} (\alpha_p (x), \alpha_p(y)) =
\inf_{\gamma \in \Gamma(x,y)} \rho(\gamma) \leq \rho(\gamma_\epsilon)  < \epsilon + KL d_X(x,y).
\end{equation}
Since $\epsilon>0$ is arbitrary, $\pdist{X} (\alpha_p (x), \alpha_p(y)) \leq KL d_X (x,y)$, as claimed.
\end{proof}

\section{Stability and Consistency for Barycentric Merge Trees} \label{S:stabcon}

The main goal of this section is to establish the stability of barycentric merge trees associated with probability measures  $\mu \in \borel{X,d_X;p}$, where the BMTs are constructed with respect to a fixed admissible pseudo metric $\fdist$. Since $\tree{X}$ is also equipped with a function $\qct_p \colon \tree{X} \to \real$ induced by the deviation function $\ct_p$ (see \eqref{E:functions}), we address stability of the full functional barycentric merge tree $\ftree{X}\coloneqq (\tree{X}, d_{p,\mm{X}}, \mu_p, \qct_p)$. In our stability and consistency results, variation in $\mu$ is quantified with the Wasserstein distance in $(X,d_X)$, from classical optimal transport theory~\cite{villani2008optimal}, whereas variation in $\ftree{X}$ is measured with a functional version of Sturm's $L_p$-transportation distance, so we begin with a discussion of this distance.


\subsection{Distances Between Functional Metric Measure Spaces}

In \cite{sturm2006}, Sturm introduced $L_p$-transportation distances, $p \geq 1$, between $mm$-spaces building on Kantorovich's formulation of optimal transport distances that are frequently referred to in the literature as the \emph{Wasserstein $p$-distances} $w_p$. Later M\'{e}moli introduced a variant, termed \emph{Gromov-Wasserstein $p$-distance}, based on expected distortions of probabilistic couplings \cite{memoli2007}. As in some of the subsequent literature Sturm's version has also been called Gromov-Wasserstein, to avoid confusion, we refer to Sturm's formulation as the Kantorovich-Sturm $p$-distance $\dks$ and denote M\'{e}moli's version $\dgw$.  The distance $\dgw$ has the virtue of being more amenable to computation and is a lower bound for $\dks$; that is, $\dgw(\mm{Z},\mm{Z}') \leq \dks(\mm{Z},\mm{Z}')$ \cite[Theorem 5.1]{memoli2011}. Because of this inequality, in studying the stability of barycentric merge trees, we use $\dks$-type metrics as they lead to stronger results that imply stability with respect to $\dgw$-type distances. 

The Kantorovich-Sturm distance can be extended to pseudo $mm$-spaces and described as follows. Let $\mm{Z} = (Z,d,\mu)$ and $\mm{Z'} = (Z',d',\mu')$ be (pseudo) $mm$-spaces, where $\mu$ and $\mu'$ have finite $p$-moments. Let $A \subseteq Z$ and $A' \subseteq Z'$ be the supports of $\mu$ and $\mu'$, respectively. A {\em metric coupling} between $\mm{Z}$ and $\mm{Z}'$ is a (pseudo) metric $\delta \colon (Z \sqcup Z') \times (Z \sqcup Z') \to \real$ on the disjoint union $Z \sqcup Z'$ with the property that $\delta|_{A\times A} = d|_{A\times A}$ and  $\delta|_{A' \times A'} = d'|_{A'\times A'}$. The set of all such metric couplings is denoted $M(\mm{Z},\mm{Z}')$. A {\em probabilistic coupling} between $\mu$ and $\mu'$ is a Borel probability measure $h$ on $Z \times Z'$ that marginalizes to $\mu$ and $\mu'$; that is, $\pi_\sharp (h) = \mu$ and $\pi'_\sharp (h) = \mu'$, where $\pi$ and $\pi'$ denote the projections to the first and second components, respectively. The set of all such probabilistic couplings is denoted $C(\mu,\mu')$.

\begin{definition}[\cite{sturm2006}] \label{D:gw}
The (extended) {\em Kantorovich-Sturm $p$-distance}, $p \geq 1$, is defined as
\[
\dks (\mm{Z},\mm{Z}') \coloneqq \inf_{\substack{h \in C(\mu,\mu')\\ \delta \in M(\mm{Z},\mm{Z}')}}
\Big( \int_{Z \times Z'} \delta^p (z,z') \,dh (z,z') \Big)^{1/p}.
\]
\end{definition}

Since a merge tree $\tree{X}$ is also equipped with a function $\qct_p \colon \tree{X} \to \real$, we adopt a functional analogue of $\dks$. For $f \colon Z \to \real$ and $f' \colon Z' \to \real$, we define a Kantorovich-Sturm distance between the functional $mm$-spaces $\mm{F} = (Z,d,\mu,f)$ and $\mm{F}' = (Z',d',\mu',f')$.

\begin{definition} \label{D:fgw}
The (extended) {\em functional Kantorovich-Sturm $p$-distance}, $p \geq 1$, is defined as
\[
\fks (\mm{F},\mm{F}') \coloneqq \inf_{\substack{h \in C(\mu,\mu')\\ \delta \in M(\mm{Z},\mm{Z}')}} \Big(\int_{Z \times Z'} \delta^p (z,z') \,dh (z,z') \Big)^{1/p} \vee \Big( \int_{Z \times Z'} |f(z)-f'(z')|^p \,dh (z,z') \Big)^{1/p}.
\]
\end{definition}
We refer to the first term in the above maximum as the {\em structural offset} of the pair $(\delta,h)$ and to the second term as the {\em functional offset} of the coupling $h$. Clearly, by their definitions,
$\dks(\mm{Z},\mm{Z}') \leq \fks(\mm{F},\mm{F}')$. The distance $\fks$ is a variant of the {\em Fused Gromov-Wasserstein distance} of \cite{vayer2020fused}, which is a functional version of $\dkm$.

Metric couplings are closely tied to relations and correspondences between $Z$ and $Z'$ (a correspondences is a relation $R \subseteq Z \times Z'$ for which the projections $\pi \colon R \to Z$ and $\pi' \colon R \to Z'$ are surjective). Here, we review a basic connection needed in the paper. Given a relation $\emptyset \ne R \subseteq Z \times Z'$, the {\em distortion} of $R$ is defined as
\begin{equation} \label{E:distortion}
dis(R) \coloneqq \sup_{(x,x'), (y,y') \in R} |d(x,y) - d'(x',y')|\,.
\end{equation}
For $r>0$, define $\delta_r \colon (Z \sqcup Z') \times (Z \sqcup Z')  \to \real$ by $\delta_r|_{Z \times Z} = d$, $\delta_r|_{Z' \times Z'} = d'$, and
\begin{equation} \label{E:deltar}
\delta_r (z,z') = \delta_r (z',z) \coloneqq r + \inf_{(w,w') \in R} d(z,w) + d'(w',z'),
\end{equation}
for any $z \in Z$ and $z' \in Z'$.

\begin{proposition}[\cite{burago2001}, see also \cite{anbouhi2024}] \label{P:r2c}
If $\emptyset \ne R \subseteq Z \times Z'$ and $dis(R) \leq 2r$, then $\delta_r$ is a pseudo metric on $Z \sqcup Z'$.
\end{proposition}

Clearly, $\delta_r (z,z') \geq r >0$, for any $z \in Z$ and $z' \in Z'$. Thus, $\delta_r$ is a metric if and only if both $d$ and $d'$ are metrics.


\subsection{Stability and Consistency}

For $\mu, \mu' \in \borel{X,d_X;p}$, we adopt the abbreviations $\mm{F}_p = (T_p, d_p, \mu_p, \qct_p)$ and $\mm{F}_p'=(T'_p, d'_p, \mu'_p, \qct'_p)$ for the functional barycentric merge trees of $\mu$ and $\mu'$, respectively, constructed with respect to a fixed admissible $\fdist \colon X \times X \to \real$. Similarly, we let $\mm{T}_p = (T_p, d_p, \mu_p)$ and $\mm{T}_p'=(T'_p, d'_p, \mu'_p)$ be the structural parts of $\mm{F}_p$ and $\mm{F}'_p$, and $\ct_p$ and $\ct_p'$ the $p$-deviation functions of $\mu$ and $\mu'$. To relate $\fks(\mm{F}_p, \mm{F}_p')$ with $w_p (\mu,\mu')$, we first construct a metric coupling between $\mm{T}_p$ and $\mm{T}'_p$ starting with the correspondence $R \subseteq T_p \times T'_p$ given by
\begin{equation}
R\coloneqq \big\{ (\alpha_p (x), \alpha'_p (x)) \colon x \in X\big\}.
\end{equation}
Given $r>0$, let $\delta_r \colon (T_p \sqcup T'_p) \times (T_p \sqcup T'_p) \to \real$ be defined as described in \eqref{E:deltar}: namely, (i) $\delta_r|_{T_p \times T_p} = d_p$; (ii) $\delta_r|_{T'_p \times T'_p} = d'_p$; and (iii) for $a \in T_p$ and $b \in T'_p$,
\begin{equation} \label{E:coupling}
\delta_r (a,b) = \delta_r (b,a)\coloneqq r + \inf_{x \in X} \big(d_p (a,\alpha_p (x)) + d'_p (\alpha'_p (x), b)\big).
\end{equation}
Note that if $x_1, x_2 \in X$, then \eqref{E:coupling} implies that
\begin{equation} \label{E:couplingbound}
\delta_r (\alpha_p (x_1),\alpha'_p (x_2)) \leq r + d_p (\alpha(x_1),\alpha_p (x_2)) ,
\end{equation}
an inequality that is used below in the proof of the stability of BMTs.

\begin{lemma}[The Coupling Lemma] \label{L:coupling}
If the $p$-deviation functions satisfy $|\ct_p (x) - \ct'_p (x)| \leq r$, $\forall x \in X$, then $\delta_r$ defines a pseudo metric on $T_p \sqcup T'_p$. In particular, $\delta_r$ yields a metric coupling between $\mm{T}_p$ and $\mm{T}'_p$.
\end{lemma}
\begin{proof}
By Proposition \ref{P:r2c}, it suffices to check that $dis (R) \leq 2r$; that is, for any $x,y \in X$,
\begin{equation} \label{E:dist}
|d_p (\alpha_p (x), \alpha_p (y)) - d'_p (\alpha'_p(x), \alpha'_p (y))| \leq 2r \,.
\end{equation}
We first show that for any curve $\gamma \in \Gamma(x,y)$,  $|\rho(\gamma) - \rho'(\gamma)| \leq 2r$, where $\rho (\gamma)$ is as in \eqref{D:predistance} and $\rho'(\gamma)$ is its counterpart for $\ct'_p$. Let $t_0 \in I$ be such that
\begin{equation}
\rho(\gamma) = \big(\ct_p (\gamma(t_0)) - \ct_p(x)\big) \vee \big(\ct_p (\gamma(t_0)) - \ct_p (y)\big).
\end{equation}
Then, by the hypothesis on the deviation functions and the triangle inequality, we have that
\begin{equation} \label{E:curve2}
\begin{split}
\ct_p(\gamma(t_0)) - \ct_p(x) &= \ct_p(\gamma(t_0)) - \ct'_p(\gamma(t_0)) + \ct'_p(\gamma(t_0)) - \ct'_p (x) + \ct'_p (x) - \ct_p(x) \\
&\leq 2r + \ct'_p(\gamma(t_0)) - \ct'_p (x) \leq 2r + \sup_{t\in I} \ct'_p(\gamma(t)) - \ct'_p (x) .
\end{split}
\end{equation}
Similarly, 
\begin{equation} \label{E:curve3}
\ct_p (\gamma(t_0)) - \ct_p (y) \leq 2r + \sup_{t\in I} \ct'_p(\gamma(t)) - \ct'_p (y)  \,.
\end{equation}
Taking the maximum of \eqref{E:curve2} and \eqref{E:curve3}, we obtain
\begin{equation}
\rho (\gamma) = \big(\ct_p (\gamma(t_0)) - \ct_p(x)\big) \vee \big(\ct_p (\gamma(t_0)) - \ct_p (y)\big) \leq 2r + \rho' (\gamma) \,.
\end{equation}
Symmetrically, we have that $\rho'(\gamma) \leq 2r + \rho (\gamma)$. Therefore, $|\rho(\gamma) - \rho'(\gamma)| \leq 2r$, as claimed. We now estimate the difference $d'_p (\alpha'_p(x), \alpha'_p (y)) - d_p (\alpha_p(x), \alpha_p (y))$. For a fixed curve $\gamma \in $, we have
\begin{equation} \label{E:curve5}
\inf_{\gamma' \in } \rho' (\gamma') - \rho(\gamma) \leq \rho'(\gamma) - \rho(\gamma) \,.
\end{equation}
Therefore,
\begin{equation}
\begin{split}
d'_p (\alpha'_p(x), \alpha'_p (y)) &- d_p (\alpha_p(x), \alpha_p (y)) = 
\inf_{\gamma' \in \Gamma(x,y)} \rho'(\gamma') - \inf_{\gamma \in \Gamma (x,y)} \rho(\gamma) \\
&= \sup_{\gamma \in \Gamma(x,y)} \big( \inf_{\gamma' \in } \rho'(\gamma') - \rho(\gamma)\big)
\leq \sup_{\gamma \in \Gamma(x,y)} \big( \rho'(\gamma) - \rho (\gamma) \big) \leq 2r.
\end{split}
\end{equation}
Similarly,
$d_p (\alpha_p(x), \alpha_p (y)) - d'_p (\alpha'_p(x), \alpha'_p (y)) \leq 2r$.
This shows that \eqref{E:dist} holds and concludes the proof.
\end{proof}
\begin{theorem}[Stability of BMTs] \label{T:stab}
Let $(X, d_X)$ be a connected and locally path-connected Polish metric space and $\fdist \colon X \times X \to \real$ an $L$-admissible pseudo metric, $L>0$. If $\mu, \mu' \in \borel{X,d_X;p}$ and $K\geq 1$ is a connectivity constant for $(X,d_X)$, then
\[
\fks(\mm{F}_p, \mm{F}'_p) \leq L(1+K) w_p (\mu,\mu').
\]
In particular, $\fks (\mm{F}_p, \mm{F}'_p) \leq 2 w_p (\mu,\mu')$ if $(X,d_X)$ is a geodesic space and $\theta=d_X$.
\end{theorem}

\begin{proof}
We begin the proof by showing that
\begin{equation} \label{E:central}
|\ct_p(x) - \ct_p'(x)| \leq L w_p (\mu,\mu') \,,
\end{equation}
for all $x \in X$, using a standard argument that employs the Minkowski inequality and the marginal conditions of a measure coupling. Indeed, for any $h \in C(\mu,\mu')$, we have that
\begin{equation}
\begin{split}
|\ct_p(x) - \ct_p'(x)| &= \Big| \Big( \int_X \fdist^p (x,y) \,d\mu (y) \Big)^{1/p} - \Big( \int_X \fdist^p (x,y') \,d\mu' (y') \Big)^{1/p} \Big| \\
&= \Big| \Big( \int_{X \times X} \fdist^p (x,y) \,dh (y,y') \Big)^{1/p} - \Big( \int_{X \times X}  \fdist^p (x,y') \,dh (y,y') \Big)^{1/p} \Big| \\
&\leq \Big(\int_{X \times X} |\fdist (x,y) - \fdist (x,y')|^p \,dh (y,y') \Big)^{1/p} \leq \Big(\int_{X \times X} \fdist^p (y,y') \,dh (y,y') \Big)^{1/p}  \\
&\leq L \Big(\int_{X \times X} d^p_X (y,y') \,dh (y,y') \Big)^{1/p} ,
\end{split}
\end{equation}
Since the coupling $h$ is arbitrary, we obtain
\begin{equation} \label{E:fdist}
|\ct_p(x) - \ct_p'(x)| \leq L \inf_{h \in C(\mu,\mu')} \Big(\int_{X \times X} d^p (y,y') \,dh (y,y') \Big)^{1/p} = L w_p (\mu,\mu'),
\end{equation}
as claimed. Therefore, the Coupling Lemma applied with $r = Lw_p (\mu,\mu')$ ensures that $\delta_r$, defined in \eqref{E:coupling}, gives a coupling between $d_p$ and $d'_p$. Also note that any $h \in C(\mu,\mu')$ induces a coupling $\bar{h} \coloneqq {(\alpha_p \times \alpha_p')}_\sharp (h) \in  C(\mu_p,\mu'_p)$, where $\mu_p = {\alpha_p}_\sharp (\mu)$ and $\mu'_p = {\alpha'_p}_\sharp(\mu')$. Then, from inequality \eqref{E:couplingbound} and Proposition \ref{P:qlip}, we obtain
\begin{equation} \label{E:estimate0}
\delta_r (\alpha_p (x), \alpha_p' (y))
\leq r + d_p (\alpha_p (x), \alpha_p (y) ) \leq L w_p(\mu,\mu') + KL d_X (x,y).
\end{equation}
Therefore, using the Minkowski inequality, we get the following estimate for the structural offset of the pair $(\delta_r,\bar{h})$:
\begin{equation} \label{E:estimate1}
\begin{split}
\Big( \int_{T_p \times T'_p} \delta_r^p (\bar{x}, \bar{y}) \, d\bar{h} (\bar{x}, \bar{y}) \Big)^{1/p} &= \Big( \int_{X \times X} \delta_r^p (\alpha_p(x), \alpha'_p(y)) \, dh (x,y) \Big)^{1/p} \\
&\leq L w_p(\mu,\mu') + KL  \Big( \int_{X \times X} d_X^p (x,y) \, dh (x,y) \Big)^{1/p}.
\end{split}
\end{equation}
For the functional offset of $\bar{h}$, \eqref{E:fdist} and Proposition \ref{P:lip} yield
\begin{equation} \label{E:estimate2}
\begin{split}
\Big( \int_{T_p \times T'_p}  |\qct_p(\bar{x}) - \qct'_p(\bar{y})|^p  \,d\bar{h}(\bar{x},\bar{y})\Big)^{1/p} 
&=\Big( \int_{X \times X} |\ct_p(x)-\ct'_p(y)|^p \, dh(x,y) \Big)^{1/p} \\
&= \Big( \int_{X \times X} |\ct_p(x)-\ct'_p(x) + \ct'_p (x)-\ct'_p(y)|^p \, dh(x,y) \Big)^{1/p} \\
&\leq L w_p (\mu,\mu') + L \Big( \int_{X \times X} d_X^p(x,y) \, dh(x,y) \Big)^{1/p},
\end{split}
\end{equation}
an upper bound that is smaller than that for the structural offset in \eqref{E:estimate1} because $K \geq 1$. Since the coupling $h$ is arbitrary, it follows that
\begin{equation}
\begin{split}
\fks (\mm{F}, \mm{F}') &\leq L w_p(\mu,\mu') + KL \inf_{h \in C(\mu,\mu')}\Big( \int_{X \times X} d_X^p (x,y) \, dh (x,y) \Big)^{1/p} \\
&= L (1+K) w_p(\mu,\mu'),
\end{split}
\end{equation}
as claimed. For a geodesic space $(X,d_X)$ and $\fdist=d_X$, we can choose $L=1$ and $K=1$, as shown in Proposition \ref{P:lip}.
\end{proof}

\begin{figure}
    \centering
    \includegraphics[width=0.9\linewidth]{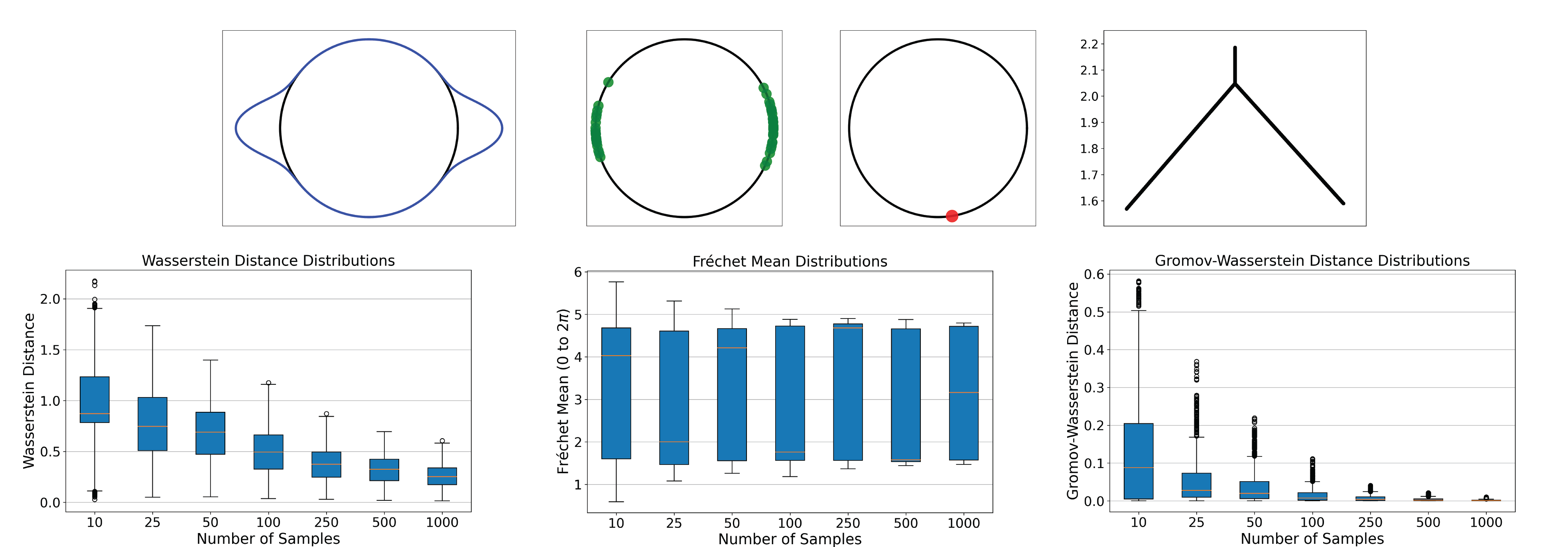}
    \caption{\small Quantitative stability of BMTs (see Example \ref{ex:quantitative}) {\bf Top Row:} A summary of the experiment. We begin with a smooth bimodal distribution on the circle, which is then sampled $n$ times ($n \in \{10,25,50,100,250,500,1000\}$), from which we compute a Fr\'echet mean (i.e., a point on the circle) and a barycentric merge tree ($p=2$). This simulation is repeated 50 times for each $n$. {\bf Bottom Row:} The distribution of Wasserstein distances between samples, treated as empirical distributions, for each value of $n$, shows convergence in number of samples. On the other hand, the Fr\'echet mean has a consistent spread for any number of samples. Finally, the distribution of Gromov-Wasserstein distances between BMTs for each $n$ shows fast convergence.}
    \label{fig:QuantitativeStability}
\end{figure}

\begin{example}[Illustrating the Stability of BMTs]\label{ex:quantitative}
    This example provides a quantitative illustration of the stability of BMTs and the lack thereof for (global) Fr\'echet means. We work with on the unit circle, endowed with geodesic distance, and with the bimodal distribution pictured in Figure \ref{fig:QuantitativeStability}. This distribution is sampled $n$ times, for $n \in \{10,25,50,100,250,500,1000\}$, from which we define the associated empirical distribution. The global Fr\'echet mean (i.e., minimizer of $\sigma_2$) and the BMT, with respect to $\sigma_2$, of this empirical distribution are then calculated, and this experiment is repeated 50 times for each value of $n$. Observe that, as $n$ increases, the Wasserstein distance between the empirical distribution and the original measure converges to zero. Despite this fact, the variance in the resulting Fr\'echet means is essentially constant---the Fr\'echet mean bounces between roughly the north pole and south pole, indicating the instability of this statistic. On the other hand, the Gromov-Wasserstein distance between the associated BMTs converges to zero rather quickly as $n$ increases---i.e., the BMTs are a stable invariant.

    We note that the Gromov-Wasserstein distance (with $p=2$) is used here, rather than the Kantorovich-Sturm distance, due to the availability of numerical solvers for the former; we use the Python Optimal Transport package for the calculation in this experiment~\cite{flamary2021pot}. Since the GW distance lower bounds the KS distance, the stability illustrated here is also guaranteed by the conclusion of Theorem \ref{T:stab}. \hfill $\Diamond$
\end{example}

For an $mm$-space $\mm{X} = (X,d_X,\mu)$, let $d_p^\ast (\mu)$ denote the upper $p$-Wasserstein dimension of $\mm{X}$ (see \cite{weed2019}).

\begin{corollary}[Empirical Estimation of FMTs] \label{C:convergence}
Let $\mm{X} = (X,d_X,\mu)$ be an $mm$-space with $\mu \in \borel{X,d_X;p}$, $p \geq 1$, where $(X,d_X)$ is a connected and locally path connected Polish space, and $(x_i)_{i=1}^\infty$ be independent samples from $\mu$. For $n>0$, let $\mu_n = \sum_{i=1}^n \delta_{x_i}/n$ be the associated empirical measure and $\mm{X}_n = (X, d_X, \mu_n)$. If $(X, d_X)$ has finite connectivity modulus and $\fdist$ is $L$-admissible, then the following hold:
\begin{enumerate}[\rm(i)]
\item (Consistency) $\lim_{n\to \infty} \fks (\ftree{X}, \mm{F} (\mm{X}_n) \to 0$,
($\otimes_\infty \mu)$-almost surely;
\item (Convergence Rate) If $s > d_p^\ast (\mu)$, then there is a constant $C>0$ such that
\[
\mathbb{E} \,[\fks (\ftree{X}, \mm{F}_p (\mm{X}_n))] \leq C {\rm diam}(X) n^{-1/s} \,,
\]
where $C$ depends only on $s$, $p$, $L$ and the connectivity modulus of $(X,d_X)$.
\end{enumerate}
\end{corollary}
\begin{proof}
(i) The statement follows from the Stability Theorem for barycentric  merge trees and the facts that $\mu_n \to \mu$ weakly a.\,s. and $w_p$ metrizes weak convergence of probability measures.

\smallskip

\noindent
(ii) This follows from the Stability Theorem and the estimate $\mathbb{E}\,[w_p(\mu,\mu_n)] \leq C' n^{-1/s}$ by Weed and Bach under the hypothesis that $\text{diam}(X) \leq 1$, where $C'$ depends only on $s$ and $p$ \cite{weed2019}.
\end{proof}


\section{Modes} \label{S:modes}

We begin our discussion of robust detection and estimation of modes of a probability distribution on a metric space $(X,d_X)$ with a review of diffusion distances on $X$ (cf.\,\cite{coifman2006}). This assumes a fixed reference measure $\nu$ on $X$; for example, the volume measure on a Riemannian manifold. 

Let $k \colon X \times X \to \real$ be a kernel function assumed to be continuous and non-negative. For each $x \in X$, define $k_x \colon X \to \real$ by $k_x(z) = k(x,z)$. Given $q \geq 1$, denote the $q$-norm on $L_q (X,\nu)$ by $\qnorm{\,\cdot \,}$. 

\begin{definition} \label{D:diffdist}
Let $k \colon X \times X \to \real$ be such that $k_x \in L_q (X,\nu)$, $\forall x \in X$. For $q \geq 1$, the {\em diffusion distance} $\ddist \colon X \times X \to \real$ is defined as
\[
\ddist (x,y) \coloneqq \qnorm{k_x-k_y} = \Big( \int_X |k(x,z) - k(y,z)|^q \,d \nu (z) \Big)^{1/q} .
\]
\end{definition}
Clearly, $\ddist$ defines a pseudo-metric on $X$ because it is the pullback of the $L_q$-distance under the map $X \to L_q(X,\nu)$ given by $x \mapsto k_x$. Sometimes we drop explicit reference to $q$ in the notation because it is fixed throughout.

\begin{definition}
Let $\mu \in \borel{X,d_X,p}$, $p\geq 1$, and assume that $\ddist$ is an admissible pseudo metric. 
\begin{enumerate}[\rm (i)]
\item A {\em $p$-mode} of $\mm{X}=(X,d_X,\mu)$ with respect to the kernel $k$ is a connected component of a local minimum set of the $p$-deviation function $\kct \colon X \to \real$ given by
\[
\kct (x) = \Big(\int_X \ddist^p (x,y)\, d\mu(y)\Big)^{1/p}.
\]
\item The {\em $p$-mode merge tree} of $\mm{X}$ for the kernel $k$ is the functional $p$-barycentric merge tree $\kftree{X}$ constructed with respect to the pseudo metric $\ddist$.
\end{enumerate}
\end{definition}

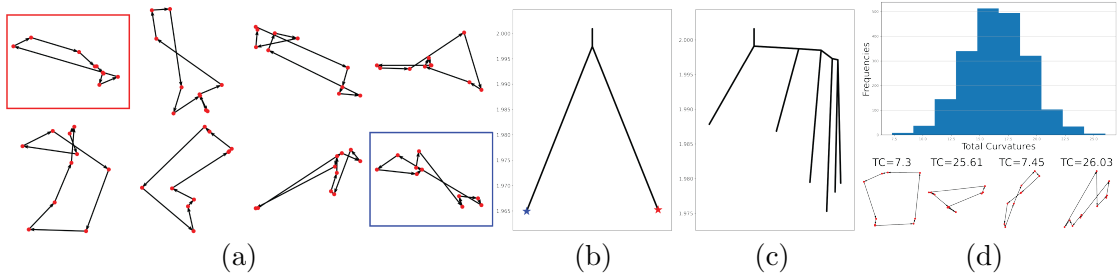
\begin{figure}
\centering
\begin{overpic}[abs,unit=1mm,width=0.9\linewidth]{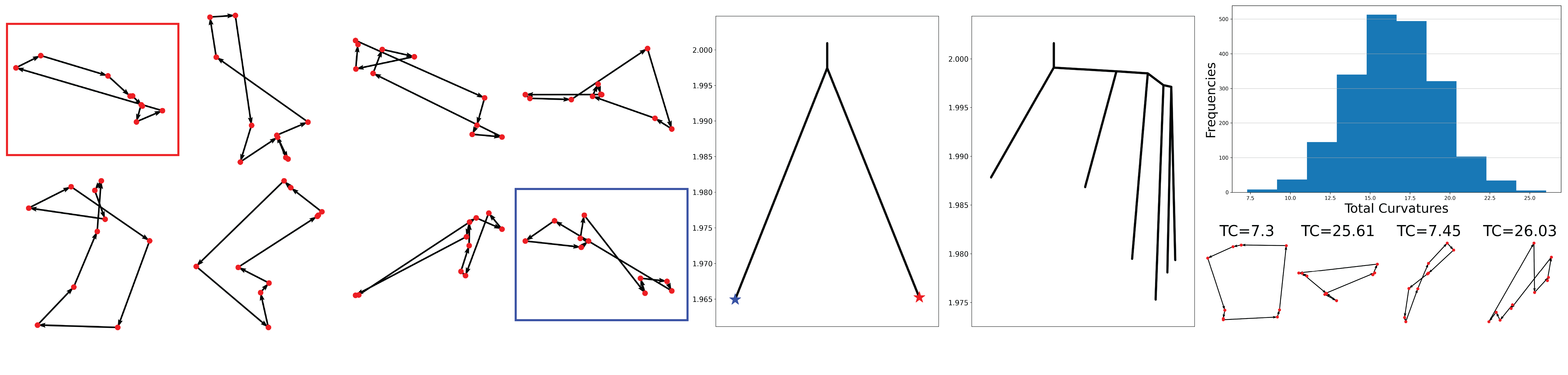} \put(29,0){(a)}
    \put(76,0){(b)}
    \put(100,0){(c)}
    \put(128,0){(d)}
    \end{overpic}
    \caption{\small Mode trees in shape space (see Example \ref{ex:polygons} for details). {\bf (a)} Random samples from the moduli space of 10-sided polygons. We work with a dataset of 2000 samples. {\bf (b)} A pdf is defined as the sum of kernels which decay exponentially in the geodesic distance from the polygons indicated by boxes. The 2-mode tree is with respect to an exponential kernel; the leaves correspond to modes, which are located at the boxed polygons. {\bf (c)} A new pdf is defined which gives higher weight to polygons with especially low or especially high total curvature. The lowest leaves of the mode tree correspond to the polygons shown in the bottom of {\bf (d)}, whose total curvatures lie in the tails of the histogram in the top of {\bf (d)}.}
    \label{fig:polygons}
\end{figure}

\begin{example}[Modes of Distributions on Shape Space]\label{ex:polygons}
    Consider the moduli space $X$ of planar $n$-gons; each point in this space is an equivalence class of a closed $n$-sided polygon, considered up to rescaling, translation and rotation. We endow $X$ with a Riemannian structure by utilizing a (perhaps surprising) correspondence between the moduli space and the Grassmann manifold $\mathrm{Gr}_2(\mathbb{R}^n)$ of 2-planes in real $n$-space~\cite{hausmann1997polygon,cantarella2019random}. The Grassmannian has a canonical Riemannian metric with an explicit formula for geodesic distance~\cite{edelman1998geometry}. Denoting the induced distance on $X$ by $d_X$, we then consider the kernel $k:X \times X \to \mathbb{R}$ defined by $k(x,y) = \exp(-10 \cdot d_X(x,y)^2)$. This kernel is used to compute mode merge trees for two distributions on the space of $10$-gons, as follows. 
    \begin{enumerate}
        \item First, we fix two (equivalence classes of) polygons $x_0,x_1 \in X$ and define a pdf on $X$ by starting with 
        $y \mapsto \exp(-d(x_0,y)^2) + \exp(-d(x_1,y)^2)$,
        then renormalizing by total volume. Intuitively, the modes of the resulting distribution should be located at the points $x_0$ and $x_1$; this intuition is borne out computationally, as is shown in Figure \ref{fig:polygons}. 
        \item Second, we define a pdf by renormalizing the function $x \mapsto (\mathrm{TC}(x) - \widehat{\mathrm{TC}})^4$, where $\mathrm{TC}(x)$ is the \emph{total curvature} of $x \in X$, or the sum of all of its turning angles, and $\widehat{\mathrm{TC}}$ is the average value of total curvature. The resulting pdf assigns heavier weights to polygons with relatively low curvature or relatively high curvature. The mode tree shown in Figure  \ref{fig:polygons} shows that the calculated local modes have extreme total curvatures values. 
    \end{enumerate}
    As was indicated in Example \ref{ex:sphere}, our numerics are performed by discretely approximating the underlying continuous spaces; in particular, we sample 2000 points from the Grassmannian by taking QR decompositions of random $10 \times 2$ matrices. We also remark that the displayed mode trees have been simplified by removing very short leaf edges, to improve visual clarity. \hfill $\Diamond$
\end{example}

By definition, the leaves of $\kftree{X}$ represent the $p$-modes of $\mm{X}$ with respect to the kernel $k$. Theorem \ref{T:stab} and Corollary \ref{C:convergence} ensure that if $(X,d_X)$ is a connected and locally path connected Polish space with finite connectivity modulus and $\ddist$ is an admissible pseudo metric, then the $p$-mode merge trees for distributions with finite $p$-moments are stable and reliably can be estimated from sufficiently large independent samples. Corollary \ref{C:convergence} also provides an estimate for the rate of convergence of mode merge trees associated with empirical measures. As such, we conclude this section by examining conditions on the kernel $k$ that guarantee the admissibility of $\ddist$.

\begin{definition}
The kernel $k \colon X \times X \to \real^+$ is {\em uniformly $L$-Lipschitz}, $L >0$, if $|k(x,z) - k(y,z)| \leq L d_X (x,y)$, for every $x,y,z \in X$.
\end{definition}

\begin{proposition} \label{P:admissible}
If the kernel $k \colon X \times X \to \real$ is uniformly $L$-Lipschitz and $\nu(X) < \infty$, then $\ddist$ is $L m_{q,\nu}$-admissible for $(X,d_X)$, where $m_{q,\nu} = (\nu(X))^{1/q}$.
\end{proposition}

\begin{proof}
The assumption that $k$ is uniform $L$-Lipschitz implies that
\begin{equation}
\ddist (x,y) = \Big( \int_X |k(x,z) - k(y,z)|^q \,d \nu (z) \Big)^{1/q} 
\leq L d_X(x,y) \Big(\int_X d\nu(z)\Big)^{1/q} = L m_{q,\nu} d_X(x,y).
\end{equation}
\end{proof}

\begin{example}
Let $(M,d_M)$ be a closed (compact and without boundary) and connected Riemannian manifold equipped with the geodesic distance $d_M$. Denote by $k_t \colon M \times M \to \real$, $t>0$, the heat kernel associated with the Laplace-Beltrami operator on $M$. By \cite[Eq.\,(2.12)]{kasue1994}, there is a constant $C_t(M)>0$ that varies continuously with $t>0$ such that
\begin{equation}
|k_t(x,z)-k_t(y,z)| \leq C_t(M) d_M (x,y),
\end{equation}
$\forall x,y,z \in M$. Thus, $k_t$ is uniformly Lipschitz and Proposition \ref{P:admissible} ensures that $\ddist$ is admissible. Since $(M,d_M)$ is a geodesic space, Theorem \ref{T:stab} guarantees that stability and consistency of mode merge trees hold for distributions on $M$.
\hfill $\Diamond$
\end{example}

Proposition \ref{P:admissible} does not apply to the heat kernel on $\real^d$ with $\nu$ the Lebesgue measure because of the finiteness condition on $\nu$. Since this is an important case for applications, particularly for $q=2$, in the next example we show by a direct calculation that, for $q=2$, the diffusion distance on $\real^d$ is admissible with respect to the Euclidean distance.

\begin{example} \label{E:diffrd}
Let $k_t \colon \real^d \times \real^d \to \real$, $t>0$, be the heat kernel that is given by $k_t (x,y) = \frac{1}{(4\pi t)^{d/2}} \exp (-\|x-y\|^2/4t)$. For $q=2$ and $t>0$, denote the corresponding diffusion distance (with $\nu$ the Lebesgue measure) by $\theta_t \colon \real^d \times \real^d \to \real$. A routine calculation (see \cite[Eq.\,6]{diaz2019probing}) shows that $\theta_t^2 (x,y) = 2 \big( \frac{1}{(8\pi t)^{d/2}} - k_{2t} (x,y)\big)$. For a fixed pair $(x,y)$, parameterize the segment from $x$ to $y$ by $\beta(s) = (1-s)x+sy$, $s \in I$, and let
\begin{equation}
g(s) \coloneqq \theta_t^2 (x, \beta(s)) =   \frac{2}{(8\pi t)^{d/2}} \left(1 - \exp \Big(-\frac{\|x-\beta(s)\|^2}{8t}\Big)\right).
\end{equation}
Then, $\theta_t^2(x,y) = g(1) - g(0)= \int_0^1 g'(s)\,ds$. The derivative of $g$ is given by
\begin{equation} \label{E:derivative}
g'(s) = \frac{(y-x)\cdot (\beta(s)-x) }{2t (8\pi t)^{d/2}} \exp \Big( -\frac{\|x-\beta(s)\|^2}{8t}\Big) \,.
\end{equation}
Since $\|\beta(s)-x\| \leq \|x-y\|$, from \eqref{E:derivative} we obtain
\[
\theta_t^2 (x,y) \leq \sup_{s\in I} |g'(s)| \leq \|x-y\|^2 / 2t (8\pi t)^{d/2}.
\]
This implies that $\theta_t (x,y) \leq \|x-y\|/\sqrt{2t} (8\pi t)^{d/4}$, showing that $\theta_t$ is $L_t$-admissible with $L_t=1/\sqrt{2t} (8\pi t)^{d/4}$.
\hfill $\Diamond$
\end{example}

Having shown in Example \ref{E:diffrd} that the diffusion distance $\theta_t$  is admissible, we now recall an interpretation given in \cite[Eq.\,10]{diaz2019probing} of the Fr\'{e}chet variance function $V_{p,t} \colon \real^d \to \real$, with respect to $\fdist_t$, for $p=2$. At scale $t/2$, we have
\begin{equation}
V_{2,t/2}(x) = \int_{\real^d} d_{t/2}^2(x,y)\,d\mu(y) = \frac{2}{(4\pi t)^{d/2}} - 2 u(t,x),
\end{equation}
where $u(t,x)$ is the solution of the heat equation $\partial_t u = \Delta u$ with initial condition $\mu$. In particular, up to a scaling factor $2$, $V_{2,t/2}$ is an upside down version of the Gaussian kernel density estimator of $\mu$ at scale $t$. Therefore, in this special case, the modes of $\mu$ defined as the connected components of the local minimum sets of $V_{2,t/2}$ coincide with the view of modes as local maxima of kernel density estimators. Note, however, that in the proof that merge trees provide a stable and consistent representation of the modes of a distribution, we do not make assumptions such as local minima occurring at isolated points. Moreover, our approach also shows how to define modes for other values of $p$ in a stable manner. For example, for $p=1$, mode merge trees can be interpreted as local-global summaries of medians of probability measures.


\section{Discrete Models} \label{S:discrete}

Corollary \ref{C:convergence} establishes the consistency of BMT estimators derived from independent samples of a distribution on a connected and locally connected Polish space $(X,d_X)$. With an eye towards computation and applications, we now discuss provably accurate combinatorial approximations to barycentric merge trees associated with empirical measures on compact spaces. We begin with a discussion of general combinatorial BMTs and then apply the construction to empirical measures on $(X,d_X)$. 

\subsection{Combinatorial Barycentric Merge Trees}

Let $G=(V,E)$ be a finite simple graph. We write (unoriented) edges as 2-element sets $\{v,w\} \in E$. In particular, loops $\{v,v\} = \{v\}$ are disallowed. We also assume that $V$ is equipped with a metric $d_V$ and a probability measure $\omega$ so that $\mm{V}=(V,d_V,\omega)$ is a finite $mm$-space. For $p \geq 1$, as before, define the $p$-deviation function $\ct_p \colon V \to \real$ by
\begin{equation}
\ct_p (v) = \Big(\sum_{v'\in V} d_V^p (v,v')\omega (v')\Big)^{1/p}.
\end{equation}
For $t \in \real$, let $G_t \subseteq G$ be the subgraph spanned by the sublevel set $\ct_p^{-1} ([0,t]) \subseteq V$; that is, the largest subgraph of $G$ whose vertex set is $\ct_p^{-1} ([0,t])$. On $V$, define an equivalence relation by $v \sim w$ if there exists $t\geq 0$ such that: (i) $\ct_p(v)=\ct_p(w)=t$ and (ii) $v$ and $w$ fall in the same connected component of $G_t$. The quotient $\qV = V/\sim$ is viewed as the vertex set of a (simple) graph $T_p = (\qV,E_p)$ whose edge set is defined by $\{[v],[w]\} \in E_p$ if $\{v,w\} \in E$ and $[v]\ne [w]$. Here, $[\,\cdot\,]$ denotes equivalence class. The quotient graph $T_p$ is the {\em $p$-barycentric merge tree} of $\mm{V}$. Clearly, $\ct_p$ induces a function $\qct_p \colon \qV \to \real$ such that $\ct_p = \qct_p \circ \alpha_p$, where $\alpha_p \colon V \to \qV$ is the quotient map. 

As in the continuous case, we introduce a poset structure and a pseudo-metric on the node set of $T_p$. For $[v] \in \qV$, let $t = \qct_p ([v])=\ct_p(v)$ and $C_v \subseteq G$ be the connected component of $G_t$ containing $v$. It is simple to verify that $C_{[v]}$ is independent of the choice of the representative $v$ of the equivalence class. Define a partial order by $[v] \preceq [w]$ if $C_v \subseteq C_w$. A vertex $[z] \in \qV$ is a merge point for $[v]$ and $[w]$ if $[v] \preceq [z]$ and $[w] \preceq [z]$. The {\em merge set} $\Lambda[v,w]$ is defined as the collection of all merge points for $[v]$ and $[w]$. The function $\gpdist \colon \qV \times \qV \to \real$, given by
\begin{equation}
\gpdist([v],[w]) = \inf_{[z] \in \Lambda[v,w]} (\ct_p (z)-\ct_p(v))\vee (\ct_p (z)-\ct_p(w)),
\end{equation}
defines a pseudo-metric on $\qV$. As in Proposition \ref{P:distance1}, we can interpret $\gpdist ([v],[w])$ in terms of (combinatorial) paths $\gamma$ connecting $v$ to $w$, where a path is given by a finite sequence $v=v_0, v_1, \ldots, v_{m-1}, v_m =w$ of vertices of $G$ such that $\{v_{i-1},v_i\} \in E$, for every $1 \leq i \leq m$. The distance $\gpdist$ may be expressed as $\gpdist ([v],[w]) =   \inf_{\gamma \in \Gamma_c(v,w)} \,\rho_c (\gamma)$, where $\Gamma_c (v,w)$ is the collection of all combinatorial paths from $v$ to $w$ and 
\begin{equation} 
\rho_c (\gamma) = \max_{v_i \in \gamma}\big(\ct_p(v_i)-\ct_p(v)\big) \vee \big(\ct_p(v_i)-\ct_p(w)\big).
\end{equation}
Setting $\omega_p \coloneqq \alpha_{p\sharp} (\omega)$, the pushforward of $\omega$ to $\qV$, we obtain a finite functional $mm$-space $(\qV,\gpdist,\omega_p, \qct_p)$.


\subsection{From Continuous to Discrete} \label{S:c2d}

Let $(X,d_X)$ be a compact, connected and locally path connected metric space. We construct discrete graph models $G=(V,E)$ of $X$ from finite $\delta$-coverings $V \subseteq X$, $\delta>0$. A subset $V=\{v_1, \ldots, v_m\} \subseteq X$ is a {\em $\delta$-covering} if for each $x\in X$, there exists $v_i \in V$ such that $d_X(x,v_i)\leq \delta$. The parameter $\delta$ is to be viewed as controlling the desired accuracy of the approximation. We denote the induced metric on $V$ by $d_V$. The edge set is defined by $\{v_i,v_j\} \in E$ if $d_V (v_i,v_j) \leq 3 \delta$ and $i\ne j$.

Given a dataset $X_n = \{x_1, \ldots, x_n\} \subseteq X$, let $p_n$ be the normalized counting measure on $X_n$, given by $p_n(x_i)=1/n$, and $\mu_n = \sum_{i=1}^n \delta_{x_i}/n$ be the associated empirical measure on $X$. Define a map $\phi \colon X_n \to V$, where $\phi(x_i)$ is arbitrarily chosen among the points $v_j \in V$ that satisfy $d_X(x_i,x_j) \leq \delta$. Define $\omega_n$ as the probability measure on $V$ given by $\omega_n (v_i) = |\phi^{-1} (v_i)|/n$, the pushforward of $p_n$ under $\phi$. We refer to the finite $mm$-space $\mm{V}_n = (V,d_V, \omega_n)$ (with the underlying graph structure described above) as a {\em combinatorial $\delta$-approximation} to $\mm{X}_n = (X,d_X, \mu_n)$. Further, if $\imath \colon V \hookrightarrow X$ denotes the inclusion map and $\nu_n = \imath_\sharp (\omega_n)$, we obtain yet another $mm$-space $\mm{W}_n = (X,d_X,\nu_n)$. We thus have three $mm$-spaces in play: (i) $\mm{X}_n = (X,d_X, \mu_n)$ which is the object of primary interest; (ii) $\mm{V}_n = (V,d_V, \omega_n)$ intended as a combinatorial proxy for $\mm{X}_n$; and (iii) $\mm{W}_n = (X,d_X,\nu_n)$ that bridges the discrete and continuous models. 

Note that, by construction, if $\cct{V} \colon V \to \real$ and $\cct{W} \colon X \to \real$ are the $p$-deviation functions of $\mm{V}_n$ and $\mm{W}_n$, then $\cct{V} (v) = \cct{W}(v)$, for any  $v \in V$. Moreover, $w_p (\mu_n,\nu_n) \leq \delta$, for any $p \geq 1$. Indeed, let $\psi \colon X_n \to X \times X$ be given by $\psi(x_i) = (x_i, \imath \circ\phi(x_i))$ and $h = \psi_\sharp (p_n)$. Then, $h \in C(\mu_n,\nu_n)$ and
\begin{equation}
w_p (\mu_n,\nu_n) \leq \Big(\int_{X \times X} d^p_X (x,y) dh(x,y)\Big)^{1/p} = \Big(\frac{1}{n}\sum_{i=1}^n  d^p_X(x_i,\phi(x_i)) \Big)^{1/p} \leq \delta.
\end{equation}
Therefore, under the assumptions of Theorem \ref{T:stab}, the functional barycentric merge trees for $\mm{X}_n$ and $\mm{W}_n$ satisfy
\begin{equation} \label{E:empcomb}
\fks \big(\cftree{X},\cftree{W}\big) \leq L(1+K)\delta \,.
\end{equation}
As such, to establish the validity of $\cftree{V}$ as an approximation to $\cftree{X}$, it suffices to obtain $\delta$-controlled upper bounds for the Kantorovich-Sturm distance $\fks \big(\cftree{V},\cftree{W}\big)$. 

Let $\cqmap{V} \colon V \to \ctree{V}$ and $\cqmap{W} \colon X \to \ctree{W}$ be the quotient maps and write the functional BMTs as $\cftree{V} = \big(\ctree{V}, \cdist{V}, \omega_{p,n}, \cqct{V}\big)$ and $\cftree{W} = \big(\ctree{W}, \cdist{W}, \nu_{p,n}, \cqct{W}\big)$, where $\omega_{p,n}$ and $\nu_{p,n}$ are the pushforwards of $\omega_n$ and $\nu_n$ under the quotient maps to the respective merge trees. As in the proof of the Stability Theorem, the key step in estimating $\fks \big(\cftree{V},\cftree{W}\big)$ is the construction of a metric coupling between $\cdist{V}$ and $\cdist{W}$, starting from a relation $R \subseteq \ctree{V} \times \ctree{W}$, which we define as
\begin{equation} \label{E:relation}
R\coloneqq \left\{\big(\cqmap{V} (v), \cqmap{W} (v)\big) \colon v\in V\right\}.
\end{equation}

\begin{lemma} \label{L:distortion}
If $K\geq 1$ is a connectivity constant for $(X,d_X)$, $\fdist$ is $L$-admissible, and $V \subseteq X$ a $\delta$-covering of $(X,d_X)$, then the distortion of the relation $R$ satisfies $dis(R) \leq KL \delta$. 
\end{lemma}

\begin{proof}
Let $v,w \in V$. Given a continuous path $\gamma \in \Gamma(v,w)$, take a grid $0=t_0 < t_1 < \ldots <t_N=1$ fine enough so that $d_X(\gamma(t_i), \gamma(t_{i-1})) \leq \delta$, for every $1 \leq i \leq N$. Let $z_0=v$, $z_N=w$, and for $0<i<N$, set $z_i \in V$ be such that $d_X(\gamma(t_i),z_i) \leq \delta$. By the triangle inequality, $d_V(z_{i-1}, z_i) \leq 3\delta$, which implies that $\{z_{i-1}, z_i\}$ is an edge of $G=(V,E)$. Define $\beta$ to be the combinatorial path $v=z_0, z_1, \ldots, z_N=w$. Then, we have
\begin{equation}
\begin{split}
\cct{V}(z_i)-\cct{V}(v) &= \cct{W}(z_i)- \cct{W}(\gamma(t_i)) + \cct{W}(\gamma(t_i)) -\cct{W}(v) \\
&\leq L d_X(z_i, \gamma(t_i)) + \sup_{t \in I} \cct{W}(\gamma(t))-\cct{W}(v) \\
&\leq L \delta + \sup_{t \in I} \cct{W}(\gamma(t))-\cct{W}(v).
\end{split}
\end{equation}
Similarly,
\begin{equation}
\cct{V}(z_i)-\cct{V}(w) \leq L \delta + \sup_{t \in I} \cct{W}(\gamma(t))-\cct{W}(w). 
\end{equation}
Therefore, 
\begin{equation}
\begin{split}
\rho_c(\beta) &= \max_i \big(\cct{V}(z_i)-\cct{V}(v) \vee \cct{V}(z_i)-\cct{V}(w)\big) \\
&\leq L\delta + \sup_{t \in I} \big(\cct{W}(\gamma(t))-\cct{W}(v)\big) \vee \big(\cct{W}(\gamma(t))-\cct{W}(w)\big) = L\delta + \rho(\gamma).
\end{split}
\end{equation}
Taking the infimum over $\gamma$, we obtain
\begin{equation} \label{E:distortion1}
\cdist{V}(\cqmap{V}(v),\cqmap{V}(w)) \leq L\delta + \cdist{W} (\cqmap{W}(v),\cqmap{W}(w)),
\end{equation}
for any $v,w\in V$. Conversely, let an edge path $\beta \in \Gamma_c(v,w)$ be given by the sequence $v=z_0, z_1, \ldots, z_N=w$ in $V$. Given $\epsilon>0$, for each $1 \leq i \leq N$, let $\gamma_i$ be a path from $z_{i-1}$ to $z_i$ with the property that 
\begin{equation}
\sup_t \big( d_X(v, \gamma_i (t)) \vee d_X(\gamma_i(t),w) \big) \leq \epsilon + K d_X(z_{i-1}, z_i),
\end{equation}
whose existence is guaranteed by the fact that $K$ is a connectivity constant for $(X,d_X)$. Concatenating all $\gamma_i$, $1 \leq i \leq N$, we obtain $\gamma \in \Gamma(v,w)$ such that for each $t \in I$, there is $i(t) \in \{0,1,\ldots,N\}$ such that $d_X(\gamma(t),z_{i(t)}) \leq \epsilon + K \delta$. Hence,
\begin{equation} \label{E:distortion2}
\begin{split}
\cct{W}(\gamma(t))-\cct{W}(v) &= \cct{W}(\gamma(t))- \cct{W}(z_{i(t)}) + \cct{V}(z_{i(t)}) - \cct{V}(v) \\
&\leq L d_X(\gamma(t), z_{i(t)}) + \max_i \cct{V}(z_i) - \cct{V}(v) \leq \epsilon + KL\delta + \rho_c(\beta).
\end{split}
\end{equation}
An identical argument shows that 
\begin{equation} \label{E:distortion3}
\cct{W}(\gamma(t))-\cct{W}(w) \leq \epsilon + KL\delta + \rho_c(\beta). 
\end{equation}
Since $\epsilon>0$ is arbitrary, taking the infimum over $\beta$ and using \eqref{E:distortion2} and \eqref{E:distortion3}, we obtain
\begin{equation} \label{E:distortion4}
\cdist{W} (\cqmap{W}(v),\cqmap{W}(w))\leq KL \delta + \cdist{V}(\cqmap{V}(v),\cqmap{V}(w)),
\end{equation}
for any $v,w \in V$. Combining \eqref{E:distortion1} and \eqref{E:distortion4}, since $K\geq 1$, we get $dis (R) \leq KL \delta$, as claimed.
\end{proof}

\begin{theorem} \label{T:approx}
If $K\geq 1$ is a connectivity constant for the compact metric space $(X,d_X)$, $\fdist$ is $L$-admissible, and $V \subseteq X$ a $\delta$-covering of $(X,d_X)$, then $\fks \big(\cftree{V},\cftree{W}\big) \leq KL\delta/2$.
\end{theorem}

\begin{proof}
Proposition \ref{P:r2c} and Lemma \ref{L:distortion} imply that, for $r=KL\delta/2$, there is a metric coupling $\delta_r \in M(\cdist{V},\cdist{W})$ such that for $\bar{v} \in \ctree{V}$ and $\bar{x} \in \ctree{W}$,
\begin{equation}
\delta_r(\bar{v},\bar{x}) = \delta_r(\bar{x},\bar{v}) = r + \inf_{w \in V} \cdist{V}(\bar{v}, \cqmap{V}(w)) + \cdist{W}(\cqmap{W}(w),\bar{x}) \,.
\end{equation} \label{E:cqbound}
In particular, $\delta_r(\cqmap{V} (v),\cqmap{W}(v)) = r$.

Let $\Delta \colon V \to V \times X$ be given by $\Delta(v) = (v,v)$. Since $\nu_n = \imath_\sharp (\omega_n)$, $h = \Delta_\sharp (\omega_n)$ gives a coupling between $\omega_n$ and $\nu_n$, which in turn, induces a coupling $\bar{h} \coloneqq {(\cct{V{}} \times \cct{W})}_\sharp (h) \in  C(\omega_{p,n},\nu_{p,n})$. Since the coupling $h$ is diagonal, by Proposition \ref{P:qlip} and \eqref{E:cqbound}, the structural offset of $(\delta_r,\bar{h})$ satisfies
\begin{equation}
\Big( \int_{\ctree{V} \times \ctree{W}} \delta_r^p (\bar{v}, \bar{x}) \, d\bar{h} (\bar{v}, \bar{x}) \Big)^{1/p} =\Big(\int_V \delta_r^p (\cqmap{V}(v),\cqmap{W}(v))\, d\omega_n (v)\Big)^{1/p} \leq r .
\end{equation}
As $\cct{V} (v) = \cct{W}(v)$, $\forall v \in V$, the functional offset of $\bar{h}$ vanishes. Indeed,
\begin{equation}
\int_{\ctree{V} \times \ctree{W}} |\cqct{V} (\bar{v}) - \cqct{W} (\bar{x})|^p \, d\bar{h} (\bar{v}, \bar{x}) = \int_V |\cct{V}(v) - \cct{W} (v)|^p \, d\omega_n (v) =0.
\end{equation}
Therefore, $\fks \big(\cftree{V},\cftree{W}\big) \leq r = KL\delta/2$.
\end{proof}

\begin{corollary}
If $K\geq 1$ is a connectivity constant for a compact metric space $(X,d_X)$, $\fdist$ is $L$-admissible, and $V \subseteq X$ a $\delta$-covering of $(X,d_X)$, then $\fks \big(\cftree{V},\cftree{X}\big) \leq C_{K,L}\,\delta$, where $C_{K,L} = L + 3KL/2$.
\end{corollary}

\begin{proof}
This follows from \eqref{E:empcomb} and Theorem \ref{T:approx}.
\end{proof}


\subsection{Binning}

Section \ref{S:c2d} constructs combinatorial approximations to barycentric merge trees with accuracy guarantees. A drawback of the construction is that, generically, the values of a $p$-deviation function on the nodes of the tree are all distinct, making the associated BMT rather large and complex with many small branches. Such merge trees do not achieve the goal of providing a concise summary of the original probability distribution. For this reason, we introduce an additional binning step to simplify the function $\cct{V} \colon V \to \real$ before constructing the merge tree. Given $\varepsilon>0$, define $\bcct \colon V \to \real$ by $\bcct(v) = \varepsilon\lfloor \cct{V}(x)/\varepsilon \rfloor$, an approximation to $\cct{V}$ that satisfies $|\cct{V}(v)-\bcct(v)| \leq \varepsilon$, for every $v\in V$. Here, $\lfloor \,\cdot\, \rfloor$ denotes the floor function. This has the effect of approximating $\cct{V}$ by a function that has constant value $i\varepsilon$ on sets of the form $\cct{V}^{-1} \big([i \varepsilon, (i+1) \varepsilon)\big) \subseteq V$, where $i$ is an arbitrary non-negative integer.

Construct a combinatorial functional merge tree $\bcftree{V}$, viewed as an approximation to $\cftree{V}$, from the same graph $mm$-structure $(V,E,d_V, \omega_n)$ only replacing $\cct{V}$ with $\bcct$. 

\begin{theorem}
For any $\varepsilon>0$, $\fks \big(\cftree{V}, \bcftree{V}\big) \leq \varepsilon$. 
\end{theorem}

\begin{proof}
We use the abbreviations $\ftree{V} = (T,d,\omega_p, \kappa)$, $\bcftree{V} = (T',d',\omega'_p, \kappa')$, $\sigma=\cct{V}$, and also $\alpha \colon V \to T$ and $\alpha' \colon V \to T'$ for the quotient maps. Since $|\sigma(v)-\bcct(v)| \leq \varepsilon$, for every $v\in V$, arguing as in Lemma \ref{L:coupling}, we have that $\delta_\varepsilon$ given on pairs $(a,b) \in T \times T'$ by
\begin{equation}
\delta_\varepsilon (a,b) = d_\varepsilon (b,a)\coloneqq \varepsilon + \inf_{v' \in V} \big(d (a,\alpha(v')) + d' (\alpha'(v'), b)\big).
\end{equation}
defines a coupling $\delta_\varepsilon \in M(d,d')$ for which $\delta_\varepsilon (\alpha(v), \alpha'(v))= \varepsilon$, $\forall v \in V$. As in the proof of Theorem \ref{T:approx}, let $h= \Delta_\sharp (\omega_n)$, where $\Delta \colon V \to V \times V$ is the diagonal map, and $\bar{h} \in C(\omega_p,\omega'_p)$ be given by $\bar{h} = (\alpha, \alpha')_\sharp (h)$. Then, the structural offset of $(\delta_\varepsilon, \bar{h})$ satisfies
\begin{equation}
\Big( \int_{T \times T'} \delta_\varepsilon^p (\bar{v}, \bar{w}) \, d\bar{h} (\bar{v}, \bar{w}) \Big)^{1/p} =\Big(\int_V \delta_\varepsilon^p (\alpha(v),\alpha'(v))\, d\omega_n (v)\Big)^{1/p} = \varepsilon .
\end{equation}
Similarly, 
\begin{equation}
\Big( \int_{T \times T'} |\kappa(\bar{v})- \kappa'(\bar{w}|^p \, d\bar{h} (\bar{v}, \bar{w}) \Big)^{1/p} =\Big(\int_V |(\sigma(v)-\bcct(v)|^p\, d\omega_n (v)\Big)^{1/p} \leq \varepsilon.
\end{equation}
Therefore, $\fks \big(\cftree{V}, \bcftree{V}\big) \leq \varepsilon$.
\end{proof}

\section{Summary and Concluding Remarks} \label{S:discussion}

As the $p$-barycenters and modes of a probability distribution $\mu$ on a metric space $(X,d_X)$ exhibit an unstable behavior, this paper introduced a functional merge tree representation of barycenters and modes that is provably stable and can be estimated reliably from data. This has been done at the generality of all Borel probability measures on a Polish metric space with finite $p$-moments. The leaves of a barycentric merge tree represent the connected components of the local minimum sets of the $p$-deviation function (or equivalently, the Fr\'{e}chet $p$-variance function) of $(X,d_X,\mu)$. The merge tree provides a sketch of the interrelationships between the various sublevel sets of the $p$-deviation function, in particular, the merging patterns of the barycenters or modes of $\mu$, thus providing a more complete summary of $(X,d_X,\mu)$ than the barycenters alone. A unifying framework was developed for the investigation of both the barycenter and mode problems, with modes viewed as special types of barycenters associated with diffusion distances. The main results are a stability theorem and a consistency result for barycentric merge trees, which were proven in a Gromov-Wasserstein type framework. The paper also presented a pathway to discretization and computation via combinatorial merge trees. The accuracy guarantees for these discrete approximations were established in Section \ref{S:c2d} only for empirical measures as this sufficed for our purposes because Corollary \ref{C:convergence} guarantees that the BMT for a more general probability measure $\mu$ can be approximated by those for empirical measures. Nonetheless, we note that the argument in Section \ref{S:c2d} can be adapted to prove a similar approximation result more directly for a general distribution $\mu$.

The choice of merge trees was largely guided by the goal of achieving stability keeping the representation of barycenters as simple as possible. Similar representations can likely be obtained through Reeb graphs (cf.\,\cite{curry2024}) at the expense of increasing the complexity of both the model and computations. The focus of the paper was on foundational aspects of stable representations of barycenters leaving questions specific to cases such as distributions on Riemannian manifolds, length spaces, networks, or Wasserstein spaces for further investigation. 

While we provided several computational examples in this paper, their primary function was to provide intuition and enhance exposition. Developing a full-fledged, efficient numerical framework for extracting BMTs from real data remains an important direction of future research. The figures in the paper can be reproduced from the open access code available at our GitHub repository (\url{https://github.com/trneedham/Barycentric-Merge-Trees}).

\bibliographystyle{abbrv}
\bibliography{mtree}

\end{document}